\documentclass[a4paper]{article}

\usepackage[pages=all, color=black, position={current page.south}, placement=bottom, scale=1, opacity=1, vshift=5mm]{background}
\SetBgContents{}

\usepackage[margin=1in]{geometry} 

\usepackage{amsmath}
\usepackage{amsthm}
\usepackage{amssymb}
\usepackage{mathrsfs} 

\usepackage{amsmath,amsfonts,bm}

\def\vzero{{\bm{0}}}

\def\ve{{\bm{e}}}

\def\vg{{\bm{g}}}

\def\vn{{\bm{n}}}

\def\vq{{\bm{q}}}

\def\vv{{\bm{v}}}

\def\vx{{\bm{x}}}

\def\mA{{\bm{A}}}

\def\mF{{\bm{F}}}

\def\mI{{\bm{I}}}

\def\mP{{\bm{P}}}

\def\mR{{\bm{R}}}

\def\mU{{\bm{U}}}
\def\mV{{\bm{V}}}

\def\mX{{\bm{X}}}

\DeclareMathAlphabet{\mathsfit}{\encodingdefault}{\sfdefault}{m}{sl}
\SetMathAlphabet{\mathsfit}{bold}{\encodingdefault}{\sfdefault}{bx}{n}

\def\gA{{\mathcal{A}}}

\def\gC{{\mathcal{C}}}
\def\gD{{\mathcal{D}}}

\def\gG{{\mathcal{G}}}

\def\gI{{\mathcal{I}}}

\def\gO{{\mathcal{O}}}

\def\gS{{\mathcal{S}}}
\def\gT{{\mathcal{T}}}

\def\sC{{\mathbb{C}}}

\def\sN{{\mathbb{N}}}

\def\sR{{\mathbb{R}}}

\newcommand{\argmin}{\mathop{\operatorname{argmin}}}
\newcommand{\argmax}{\mathop{\operatorname{argmax}}}

\newcommand{\HT}{{\mathrm{H}}} 

\newcommand{\distortho}{{\mathrm{d}}_{\orthgroup/\gG}}
\newcommand{\distsph}{{\mathrm{d}}_{\sph/\gG}}
\newcommand{\distort}{{\mathrm{d}}_{\orthgroup}}
\newcommand{\distsp}{{\mathrm{d}}_{\sph}}

\DeclareMathOperator{\Tr}{Tr}
\DeclareMathOperator{\Mean}{Mean}
\DeclareMathOperator{\Var}{Var}
\DeclareMathOperator{\Conv}{Conv}
\newcommand{\SO}[1]{\mathrm{SO}(#1)}
\newcommand{\orthgroup}{{\SO{3}}} 
\newcommand{\sph}{{\mathrm{S}^2}} 

\usepackage[utf8]{inputenc}
\usepackage{hyperref}
\hypersetup{
	unicode,
	pdfproducer={LaTeX},
	pdfcreator={pdflatex}
}

\theoremstyle{plain}
\newtheorem{theorem}{Theorem}

\newtheorem{proposition}[theorem]{Proposition}

\theoremstyle{definition}

\usepackage{graphicx, color}
\graphicspath{{figures/}}

\usepackage{algorithm, algpseudocode} 

\usepackage{booktabs}
\usepackage{multirow}
\usepackage{multicol}
\usepackage{lscape}

\title{Averaging Orientations with Molecular Symmetry in Cryo-EM}
\author{
	Qi Zhang$^1$ $^2$ $^3$ $^4$ \and
	Chenglong Bao$^5$ $^6$ \thanks{Chenglong Bao is supported by the National Key R\&D Program of China (No.2021YFA1001300), National Natural Science Foundation of China (No.12271291), Tsinghua University Initiative Scientific Research Program. (clbao@mail.tsinghua.edu.cn)} \and
	Hai Lin$^7$ $^8$ \thanks{Hai Lin is supported by the National Key R\&D Program of China (No. 2020YFA0713000) and grant TH-533310008 of Tsinghua University. (hailinhl@seu.edu.cn, hailin@mail.tsinghua.edu.cn)} \and
	Mingxu Hu$^9$ $^4$ $^3$ \thanks{Mingxu Hu is supported by Shenzhen Medical Academy of Research and Translation, Beijing Frontier Research Center for Biological Structure (Tsinghua University) and Advanced Innovation Center for Structural Biology (Tsinghua University). (humingxu@mail.tsinghua.edu.cn)}}
\date{
	$^1$Key Laboratory of Protein Sciences (Tsinghua University), Ministry of Education, Beijing, China \\
	$^2$School of Life Science, Tsinghua University, Beijing, China \\
	$^3$Beijing Advanced Innovation Center for Structural Biology (Tsinghua University), Beijing, China \\
	$^4$Beijing Frontier Research Center for Biological Structure (Tsinghua University), Beijing, China \\
	$^5$Yau Mathematical Sciences Center, Tsinghua University, Beijing, China \\
	$^6$Yanqi Lake Beijing Institute of Mathematical Sciences and Applications, Beijing, China \\
	$^7$Shing-Tung Yau Center of Southeast University, Southeast University, Nanjing, China \\
	$^8$School of Mathematics, Southeast University, Nanjing, China \\
    $^9$Institute of Bio-Architecture and Bio-Interactions (IBABI), Shenzhen Medical Academy of Research and Translation, Shenzhen, China
}

\begin{document}

\maketitle

\begin{abstract}
Cryogenic electron microscopy (cryo-EM) is an invaluable technique for determining high-resolution three-dimensional structures of biological macromolecules using transmission particle images. The inherent symmetry in these macromolecules is advantageous, as it allows each image to represent multiple perspectives. However, data processing that incorporates symmetry can inadvertently average out asymmetric features. Therefore, a key preliminary step is to visualize 2D asymmetric features in the particle images, which requires estimating orientation statistics under molecular symmetry constraints. Motivated by this challenge, we introduce a novel method for estimating the mean and variance of orientations with molecular symmetry. Utilizing tools from non-unique games, we show that our proposed non-convex formulation can be simplified as a semi-definite programming problem. Moreover, we propose a novel rounding procedure to determine the representative values. Experimental results demonstrate that the proposed approach can find the global minima and the appropriate representatives with a high degree of probability. We release the code of our method as an open-source Python package named pySymStat. Finally, we apply pySymStat to visualize an asymmetric feature in an icosahedral virus, a feat that proved unachievable using the conventional 2D classification method in RELION. \\

\noindent\textbf{Keywords:} Cryo-EM, orientation estimation, averaging over $\orthgroup$ and $\sph$, molecular symmetry, non-unique games \\

\noindent\textbf{AMS subject classifications:} 65K05, 90C26, 65Z05
\end{abstract}

\section{Introduction}

Structural biology, investigating the three-dimensional structures of biological macromolecules, offers direct observations that facilitate insights into the structures and functions of these macromolecules. Among various imaging techniques, cryogenic electron microscopy (cryo-EM) has emerged as a leading tool in structural biology. It allows the determination of near-atomic resolution 3D structures of biological macromolecules in a relatively cost-effective and time-efficient manner~\cite{Cheng_2018_SPR_Review}. Such has been the impact of cryo-EM, that it was selected as the ``Method of the Year 2015'' by \textsl{Nature Methods}, and three pioneers in the field were awarded the Nobel Prize in Chemistry in 2017.

The primary steps in cryo-EM include sample preparation, image processing, and atomic model building. During the sample preparation phase, solutions containing target biological macromolecules are rapidly frozen to produce amorphous thin films, a process known as vitrification. Images are captured using transmission electron microscopy (TEM). From these images, individual particle images, each containing a target biological macromolecule, are extracted. Subsequently, using the gathered 2D particle images, cryo-EM image processing seeks to reconstruct high-resolution 3D density maps. These maps are then utilized to build atomic models of the target biological macromolecule. Mathematically, let $X:\mathbb{R}^3\to\mathbb{R}$ be the 3D density map to be estimated and $\{I_i:\mathbb{R}^2\to\mathbb{R}\}_{i=1}^N$ be a set of $N$ transmission particle images (also known as samples), the physical model in cryo-EM is represented by~\cite{scheres_relion_2012}
\begin{equation}\label{2Dprojections}
    I_i = h_i \ast T_i \circ P_z(\mR_i X) + n_i, \quad i = 1,2,\cdots,N,
\end{equation}
where $h_i$ represents the contrast transfer function (CTF) resulting from the electron microscope's lens system~\cite{CTFFIND4_2015}. The symbol $\ast$ denotes the convolution operator, $T_i$ stands for the in-plane translation operator, and $P_z$ is the projection operator along the $z$-axis. The term $\mR_i$ from $\orthgroup = \{\mR \in \mathbb{R}^{3\times 3} | \mR^\top\mR = I, \det(\mR) = 1\}$ corresponds to the pose of the $i$-th particle, and $n_i$ is the noise. Specifically, the CTF describes how the electron microscope optics modulate the image contrast across various spatial frequencies. While it is accurately determined during the preliminary step of the cryo-EM image processing workflow, it is not flawless. Pose estimation pertains to deducing the 3D orientation of particles within cryo-EM images, encompassing three Euler angles and two translational parameters. These are denoted as $\mR_i$ and $T_i$ in the model, respectively. Based on the estimated parameters, the 3D density map $X$ can be derived by solving \eqref{2Dprojections}. Typically, cryo-EM image processing operates in an iterative cycle between parameter estimation and reconstruction. Even with the emergence of cryo-EM image processing software like RELION~\cite{scheres_relion_2012}, CryoSPARC~\cite{punjani_cryosparc_2017}, and cisTEM~\cite{grant_cistem_2018}, image processing remains challenging due to the extremely low signal-to-noise ratio. Among them, we introduce an unsolved computational issue in cryo-EM in the following context.

\textbf{Symmetry mismatch issue.} Many biological molecules exhibit inherent symmetry. For example, viruses often possess icosahedral symmetry. Such symmetry can be leveraged to enhance reconstruction by averaging symmetry-related views. However, this approach assumes absolute symmetry and averages out any asymmetric features, giving rise to what is termed the symmetry mismatch issue~\cite{li_symmetry-mismatch_2016}. The structural study of icosahedral viruses exemplifies this issue well~\cite{liu_cryo-em_2015}. Asymmetric structural elements, including the genome, minor structural proteins, and interactions with the host during the viral life cycle, are averaged out. These elements are pivotal to processes like viral infection, replication, assembly, and transmission~\cite{lee_transferrin_2019, stevens_asymmetric_2023}. Hence, gleaning detailed insights into these asymmetric features is paramount for a comprehensive understanding of viral behaviors. To elaborate on the symmetry mismatch issue, we define $\gG\subset\orthgroup$ to be a molecular symmetry group. The density map $X$ can be expressed as $X = X_{asym} + X_{sym}$, with $X_{asym}$ representing the asymmetric component and $X_{sym}$ denoting the symmetric component such that $\vg X_{sym} = X_{sym}$ for all $\vg \in \gG$. For each particle image $I_i$, current algorithms determine its pose parameters $\mR_i$ by
\begin{equation}\label{eqn:loss}
    \min_{\mR_i\in\mathrm{SO}(3)} \mathrm{loss}(\mathcal{A}(X;\mR_i), I_i) = \mathrm{loss}(\mathcal{A}(X_{sym};\mR_i)+\mathcal{A}(X_{asym};\mR_i), I_i)
\end{equation}
where $\mathcal{A}$ is the forward linear operator in cryo-EM parameterized by $\mR_i$ and $\mathrm{loss}$ is the loss function (See more details in~\cite{scheres_relion_2012, punjani_cryosparc_2017, grant_cistem_2018, hu_particle-filter_2018}). Since $X_{sym}$ is the dominant part in $X$, the strength of $\mathcal{A}(X_{asym};\mR_i)$ is negligible in compassion with the first term and the high noise in $I_i$, and we have
\begin{equation}
    \begin{aligned}
    \mathrm{loss}(\mathcal{A}(X_{sym};\mR_i)+\mathcal{A}(X_{asym};\mR_i), I_i) & \approx\mathrm{loss}(\mathcal{A}(X_{sym};\mR_i), I_i) \\
    & = \mathrm{corr}(\mathcal{A}(X_{sym};\mR_i\vg), I_i),~\forall \vg\in\mathcal{G}.
    \end{aligned}
\end{equation}
In practice, we can not distinguish the elements in $[\mR_i]$. To find the correct pose, we have to find a unique $\vg_i \in \mathcal{G}$ such that $\mR_i \vg_i$ is the real pose parameter of $X$, which is a challenging task known as the symmetry mismatch issue. In scenarios involving symmetry mismatch, our focus is on determining the set \(\{\vg_i\} \subset \gG\), assuming \(\{[\mR_i]\}\) are given.

\textbf{Pre-step: 2D asymmetric feature visualization.} Before solving this, visualizing asymmetric features by averaging images $I_i$ within a cluster is important. This helps determine, for instance,  whether the asymmetric feature genuinely exists in the samples $\{I_i\}$, whether it has been adequately stabilized, and if the asymmetric feature exhibits specific positional characteristics (such as binding to the 2-fold axis, 3-fold axis, or 5-fold axis of an icosahedral virus). In other words, the visualization process needs to cluster ${I_i}$ based on the given rotation matrices $\{[\mR_i]\}$ and the symmetric property $\gG$ in $X_{sym}$, then average within each cluster. Similar to the previously defined $[\mR_i]$, let $[\vn_i] = \{\vg^\top\vn_i \mid \vg \in \gG\} \in \sph/\gG$, where \(\vn_i \in \sph\), be the projection direction corresponding to \(\mR_i\). Here, \(\sph\) refers to the unit sphere in \(\mathbb{R}^3\), defined as \(\sph = \{\vn \mid \|\vn\|_2 = 1\}\). The computational challenge in the clustering step is: \emph{How can we compute the mean or variance of \(\{[\mR_i]\}\) and \(\{[\vn_i]\}\) on the quotient manifold \(\orthgroup/\gG\) and \(\sph/\gG\), respectively?}

The above problem relates to a challenging discrete optimization problem, posing significant computational difficulties. Motivated by this problem, our main contributions are summarized as follows:
\begin{itemize}
    \item  We first approximate the variance calculation by using the pairwise distance of spatial rotations and projections, which requires computing the proper representatives $\vg_i\in\gG$ for $\mR_i$ and $\vn_i$, $i=1,2,\ldots, N$. Moreover, we establish the approximation error between the original and the approximated version if the distances in $\orthgroup/\gG$ and $\sph/\gG$ are induced from the corresponding Euclidean norm.
    \item Since the approximated version is still non-convex, we provide a convex relaxation for estimating the empirical mean and variance on $\orthgroup/\gG$ and $\sph/\gG$ using the non-unique games (NUG) framework~\cite{Singer_2015_NUG, Lederman_2016} and representation theory of $\gG$. Additionally, we propose a new rounding algorithm to obtain the final solution and have released an open-source Python package, pySymStat (\url{https://github.com/mxhulab/pySymStat}).
    \item Extensive results on various molecular symmetry groups $\gG$ (including cyclic group $\gC_n$, dihedral group $\gD_n$, tetrahedral group $\gT$, octahedral group $\gO$ and icosahedral group $\gI$) demonstrate that our method achieves the global optimum with high probability. Finally, we applied pySymStat to visualize the 2D asymmetric feature in an icosahedral virus in a synthetic dataset, a feat that proved unachievable using 2D classification in RELION.
\end{itemize}

\section{Problem formulation}

We first give the distance on $\orthgroup/\gG$ and $\sph/\gG$ and then formulate the mathematical optimization models related to the mean and variance estimations on the above two quotient manifolds. At the end of this section, we present the corresponding non-convex relaxations and establish the relationship to the original problem.

\subsection{Distances of quotient manifolds}

Let $\distort$ be a distance on $\orthgroup$ and $\distsp$ be a distance on $\sph$. Assume that $\distort$ and $\distsp$ are $\orthgroup$-invariant, i.e.,
\begin{equation}
\begin{aligned}
    & \distort(\mR_1,\mR_2) = \distort(\mR_1\mR,\mR_2\mR), \quad \forall\mR_1,\mR_2,\mR\in\orthgroup, \\
    & \distsp(\vn_1,\vn_2) = \distsp(\mR^\top\vn_1,\mR^\top\vn_2), \quad \forall\vn_1,\vn_2\in\sph,\forall\mR\in\orthgroup.
\end{aligned}
\end{equation}
We define the distances on quotient manifolds $\orthgroup/\gG$ and $\sph/\gG$ as
\begin{equation}
\label{dist:quotient}
\begin{aligned}
    \distortho([\mR_1],[\mR_2]) & = \min_{\vg_1,\vg_2\in\gG}\distort(\mR_1\vg_1,\mR_2\vg_2) = \min_{\vg\in\gG}\distort(\mR_1\vg,\mR_2), \\
    \distsph([\vn_1],[\vn_2]) & = \min_{\vg_1,\vg_2\in\gG}\distsp(\vg_1^\top\vn_1,\vg_2^\top\vn_2) = \min_{\vg\in\gG}\distsp(\vg^\top\vn_1,\vn_2),
\end{aligned}
\end{equation}
where the last equalities of the above two formulas follow immediately from the $\orthgroup$-invariance. This statement is proved in Section 1 of the supplementary material, where a similar proof can be found in \cite{Functional_and_Shape_Data_Analysis}.

Now, we are ready to show the well-definedness of \eqref{dist:quotient}.
\begin{proposition}
The following two statements hold:
\begin{itemize}
    \item The function $\distortho$ defined in \eqref{dist:quotient} gives a distance on $\orthgroup/\gG$.
    \item The function $\distsph$ defined in \eqref{dist:quotient} gives a distance on $\sph/\gG$.
\end{itemize}
\end{proposition}
\begin{proof}
We present the proof of the first statement in section 1 of the supplementary material. The proof of the second statement is similar and therefore omitted.
\end{proof}

There are two typical $\orthgroup$-invariant distances on $\orthgroup$. One is the arithmetic distance~\cite{Hu_2019}, i.e.,
\begin{equation*}
    \distort^A(\mR_1,\mR_2) = \|\mR_1-\mR_2\|_F,
\end{equation*}
where $\|\cdot\|_F$ is the Frobenius norm. The other is geometric distance~\cite{Hu_2019}, that is
\begin{equation*}
    \distort^G(\mR_1,\mR_2) = \max_{\vv\in\sph}\{\cos^{-1}(\mR_1\vv\cdot\mR_2\vv)\}.
\end{equation*}

Similarly, we can consider the arithmetic distance $\distsp^A$ and geometric distance $\distsp^G$ on $\sph$, defined as
\begin{equation*}
    \distsp^A(\vn_1,\vn_2) = \|\vn_1-\vn_2\|_2,\quad\quad \distsp^G(\vn_1,\vn_2) = \cos^{-1}(\vn_1\cdot\vn_2),
\end{equation*}
where $\|\cdot\|_2$ is the $\ell_2$-norm of a vector. In the next part, we use $\distortho^A, \distortho^G$ to be distances induced by $\distort^A, \distort^G$ respectively and $\distsph^A, \distsph^G$ to be the distances induced by $\distsp^A, \distsp^G$ respectively.

\subsection{Mean and variance on quotient manifolds}

Let $\{[\mR_i]\}_{i=1}^N$ and $\{[\vn_i]\}_{i=1}^N$ be a set of spatial rotations and projection directions respectively, we define the following two problems
\begin{align}
    & \min_{[\mR]\in\orthgroup/\gG}\frac{1}{N}\sum_{i=1}^N\distortho([\mR],[\mR_i])^2, \label{formulation:rotation} \\
    & \min_{[\vn]\in\sph/\gG}\frac{1}{N}\sum_{i=1}^N\distsph([\vn],[\vn_i])^2. \label{formulation:projection}
\end{align}
The mean and variance of spatial rotations with molecular symmetry $\gG$, denoted as $\Mean(\{[\mR_i]\})$ and $\Var(\{[\mR_i]\})$, are the optimal solution and optimal value of \eqref{formulation:rotation}, respectively. Similarly, $\Mean(\{[\vn_i]\})$ and $\Var(\{[\vn_i]\})$ are defined to be the optimal solution and optimal value of \eqref{formulation:projection}, respectively.

By \eqref{dist:quotient} and defining
\begin{align}
    & L^\orthgroup(\vg_1,\vg_2,\cdots,\vg_N) := \min_{\mR\in\orthgroup}~\left\{\frac{1}{N}\sum_{i=1}^N\distort(\mR,\mR_i\vg_i)^2\right\}, \label{L-orthgroup}\\
    & L^\sph(\vg_1,\vg_2,\cdots,\vg_N) := \min_{\vn\in\sph}~\left\{\frac{1}{N}\sum_{i=1}^N\distsp(\vn,\vg_i^\top\vn_i)^2\right\}, \label{L-sph}
\end{align}
\eqref{formulation:rotation} and \eqref{formulation:rotation} can be simplified to
\begin{align}
    & \min_{\vg_1,\vg_2,\cdots,\vg_N\in\gG}L^\orthgroup(\vg_1,\vg_2,\cdots,\vg_N), \label{formulation:rotation-new}\\
    & \min_{\vg_1,\vg_2,\cdots,\vg_N\in\gG}L^\sph(\vg_1,\vg_2,\cdots,\vg_N).\label{formulation:projection-new}
\end{align}
Once the optimal representatives are obtained, we can determine the mean of spatial rotations and projection via minimizing \eqref{L-orthgroup} and \eqref{L-sph}. However, due to the discrete nature in \eqref{formulation:rotation-new} and \eqref{formulation:projection-new}, directly solving them is challenging. Thus, we would like to solve them approximately.

The approximated version of $L^\orthgroup$ and $L^\sph$ is made based on the following observation in Euclidean space. Let $\{\vx_i\}\subseteq\mathbb{R}^n$ be a set of points, the empirical mean $\bar \vx$ is
\begin{equation}
    \bar \vx := \Mean(\{\vx_i\}) = \argmin_{\vx} \frac{1}{N}\sum_{i=1}^N\|\vx-\vx_i\|_2^2 = \frac{1}{N}\sum_{i=1}^N\vx_i
\end{equation}
and the variance ${\rm Var}(\{\vx_i\})$ is
\begin{equation}\label{variance:Euclidean}
    {\rm Var}(\{\vx_i\}) = \min_{\vx} \frac{1}{N}\sum_{i=1}^N\|\vx-\vx_i\|_2^2 = \frac{1}{N}\sum_{i=1}^N\|\bar\vx - \vx_i\|_2^2 = \frac{1}{2N^2}\sum_{i,j=1}^N\|\vx_i-\vx_j\|_2^2.
\end{equation}
Equation \eqref{variance:Euclidean} implies that the variance can be obtained via the pairwise distance of $\{\vx_i\}_{i=1}^N$, without solving $\min \frac{1}{N}\sum_{i=1}^N\|\vx-\vx_i\|_2^2$. Thus, we generalize this idea to $\orthgroup$ and $\sph$, i.e., we approximate the $L^\orthgroup$ and $L^\sph$ via the pairwise distance
\begin{align}
    & \tilde{L}^\orthgroup(\vg_1,\vg_2,\cdots,\vg_N) = \frac{1}{2N^2} \sum_{i,j=1}^N\distort(\mR_i\vg_i,\mR_j\vg_j)^2, \label{L-orthgroup:approx}\\
    & \tilde{L}^\sph(\vg_1,\vg_2,\cdots,\vg_N) = \frac{1}{2N^2}\sum_{i,j=1}^N\distsp(\vg_i^\top\vn_i,\vg_j^\top\vn_j)^2. \label{L-sph:approx}
\end{align}
In this case, the problems \eqref{formulation:rotation-new} and \eqref{formulation:projection-new} have the approximations:
\begin{align}
    & \min_{\vg_1,\vg_2,\cdots,\vg_N\in\gG}~\tilde{L}^\orthgroup(\vg_1,\vg_2,\cdots,\vg_N) \label{formulation:rotation-approx}\\
    & \min_{\vg_1,\vg_2,\cdots,\vg_N\in\gG}~\tilde{L}^\sph(\vg_1,\vg_2,\cdots,\vg_N).\label{formulation:projection-approx}
\end{align}
Once the solution of \eqref{formulation:rotation-approx} and \eqref{formulation:projection-approx}, denoted as $\{\vg_i^{\orthgroup}\}$ and $\{\vg_i^{\sph}\}$ respectively, are obtained, the variances are estimated via
\begin{align}
    & \widetilde{\Var}(\{[\mR_i]\}) = L^\orthgroup(\vg_1^{\orthgroup},\vg_2^{\orthgroup},\cdots,\vg_n^{\orthgroup}), \label{approx:var-rotation}\\
    & \widetilde{\Var}(\{[\vn_i]\}) = L^\sph(\vg_1^{\sph},\vg_2^{\sph},\cdots,\vg_n^{\sph}), \label{approx:var-projection}
\end{align}
respectively, and the means are also estimated as
\begin{align}
    & \widetilde{\Mean}(\{[\mR_i]\}) = [\Mean\{\mR_i\vg_i^{\orthgroup}\}], \\
    & \widetilde{\Mean}(\{[\vn_i]\}) = [\Mean\{(\vg_i^{\sph})^\top\vn_i\}],
\end{align}
where $\Mean$ in the right hand side is the ordinary mean on $\orthgroup$ and $\sph$, without considering molecular symmetry, which can be solved by existing approaches~\cite{Horn_1987, Directional_Statistics}. Consequently, the computational bottleneck is solving \eqref{formulation:rotation-approx} and \eqref{formulation:projection-approx}, which are generally challenging. Thanks to the recently developed non-unique games~(NUG) framework~\cite{Singer_2015_NUG}, we relax~\eqref{formulation:rotation-approx} and \eqref{formulation:projection-approx} to positive semi-definite programming that existing convex optimization methods can optimize. Before presenting the numerical algorithm for solving \eqref{formulation:rotation-approx} and \eqref{formulation:projection-approx}, we give more explanations that show the rationality of our approximation.

\subsection{The analysis of the approximated model}

In this section, choosing the arithmetic distances in $\orthgroup$ and $\sph$, we analyze the errors of the approximated mean and variance, which are summarized as the next two theorems.

\begin{theorem}\label{thm:sph-approx}
Given $\{\vn_i\}_{i=1}^N\subset\sph$ and assume $\distsp=\distsp^A$, we have
\begin{equation}
    \tilde{L}^\sph(\vg_1,\cdots,\vg_N) = f(L^\sph(\vg_1,\cdots,\vg_N)),\quad\forall\vg_1,\cdots,\vg_N\in\gG,
\end{equation}
where $f(x)=x-\frac{1}{4}x^2$. In particular,
\begin{equation}
\label{projection-equivalence}
    \Mean(\{[\vn_i]\}) = \widetilde{\Mean}(\{[\vn_i]\}),\quad \Var(\{[\vn_i]\}) = \widetilde{\Var}(\{[\vn_i]\}).
\end{equation}
\end{theorem}
\begin{proof}
Write $L^\sph(\vg_1,\vg_2,\cdots,\vg_N)$ and $\tilde{L}^\sph(\vg_1,\vg_2,\cdots,\vg_N)$ as $L^\sph$ and $\tilde{L}^\sph$ for short. Let $\tilde\vn = \frac{1}{N}\sum_{i=1}^N\vg_i^\top\vn_i$. By \eqref{variance:Euclidean}, we know
\begin{equation*}
    \tilde{L}^\sph = \frac{1}{N} \sum_{i=1}^N\|\tilde\vn - \vg_i^\top\vn_i\|_2^2 = 1-\|\tilde\vn\|_2^2.
\end{equation*}
By direct calculation, it has
\begin{equation*}
    L^\sph = \frac{1}{N}\sum_{i=1}^N\|\frac{1}{\|\tilde\vn\|_2}\tilde\vn-\vg_i^\top\vn_i\|_2^2 = 2 (1-\|\tilde\vn\|_2).
\end{equation*}
Thus, we know
\begin{equation*}
    \tilde{L}^\sph = L^\sph - \frac{(L^\sph)^2}{4}.
\end{equation*}
Since $L^\sph\in[0,2]$ and $f$ is monotone increasing on $[0,2]$ we have $\argmin\limits_{\vg_1,\cdots,\vg_N\in\gG}L^\sph = \argmin\limits_{\vg_1,\cdots,\vg_N\in\gG}\tilde{L}^\sph$, which completes the proof.
\end{proof}

\begin{figure}[htbp]
    \centering
    \includegraphics[width=\textwidth]{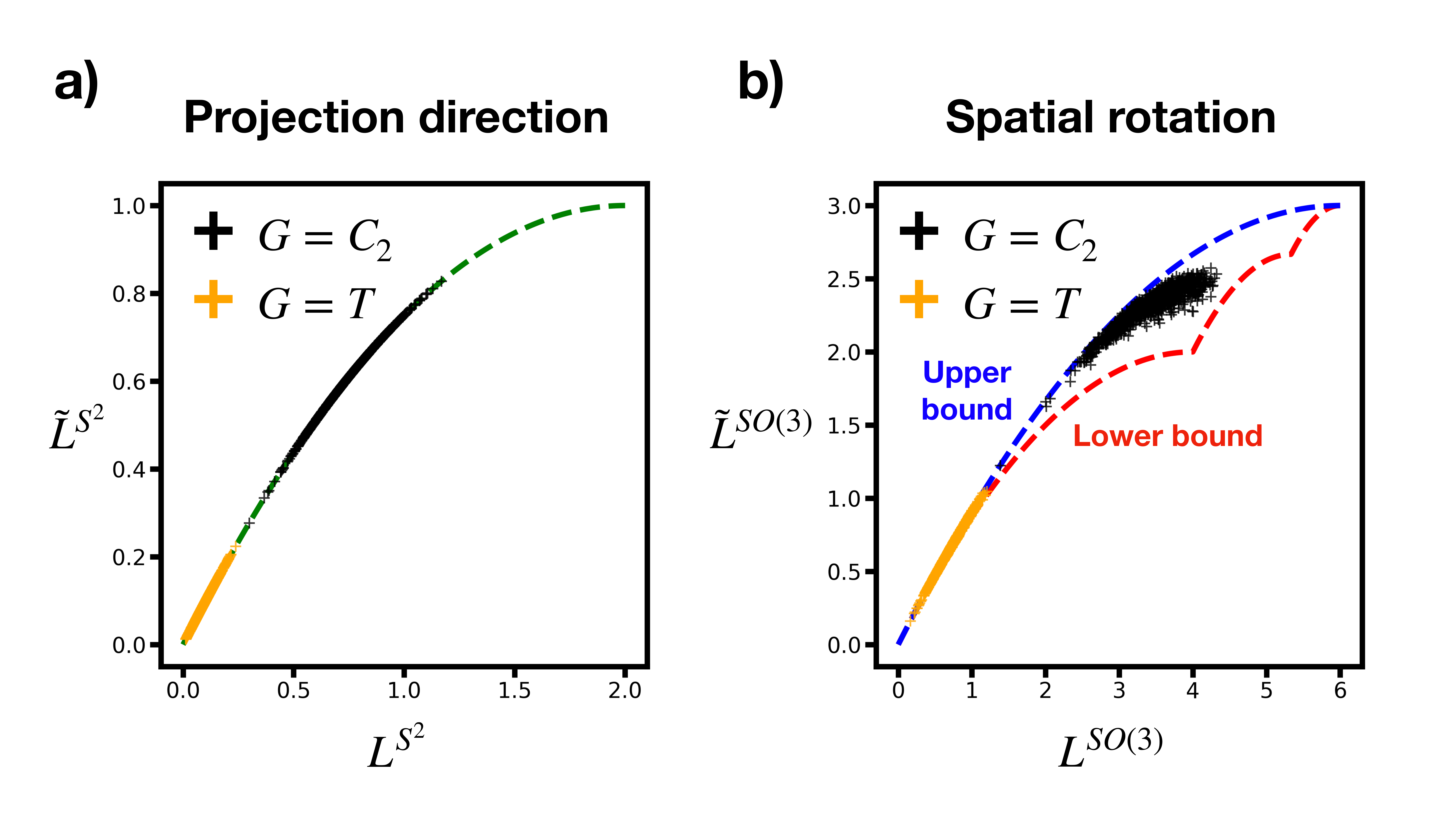}
    \caption{\textbf{The analysis of the approximated model.} \textbf{a}, The green curve is the theoretical relation between $L^{\sph}$ and its approximation $\tilde{L}^{\sph}$ (see \textbf{Theorem} \ref{thm:sph-approx}). \textbf{b}, The red and blue curve are theoretical upper and bound of $\tilde{L}^{\orthgroup}$ with given $L^{\orthgroup}$ (see \textbf{Theorem} \ref{thm:ortho-approx}), respectively. The numerical experiments in the case $\gG=\gC_2,\gT$ are presented by black and orange crosses, respectively (see \textbf{Section} \ref{sec:simulation_experiments}).}
    \label{fig:range}
\end{figure}

\begin{theorem}\label{thm:ortho-approx}
Given $\{\mR_i\}_{i=1}^N\subset\orthgroup$ and assume $\distort = \distort^A$, we have
\begin{equation}
    f_1(L^\orthgroup(\vg_1,\cdots,\vg_N)) \leq \tilde{L}^\orthgroup(\vg_1,\cdots,\vg_N) \leq f_2(L^\orthgroup(\vg_1,\cdots,\vg_N))
\end{equation}
for all $\vg_1,\cdots,\vg_N\in\gG$, where
\begin{equation}\label{equivlance_rotations_bound}
    f_1(x) =
    \begin{cases}
        x-\frac{1}{8}x^2, & \mbox{ if }\quad x\in[0,4], \\
        -8+4x-\frac{3}{8}x^2, & \mbox{ if }\quad x\in[4,\frac{16}{3}], \\
        -24+9x-\frac{3}{4}x^2, & \mbox{ if }\quad x\in[\frac{16}{3},6],
    \end{cases}
    \quad\quad
    f_2(x) = x-\frac{1}{12}x^2.
\end{equation}
In particular,
\begin{equation}
    \Var(\{[\mR_i]\}) \geq f_2^{-1}(f_1(\widetilde{\Var}(\{[\mR_i]\}))).
\end{equation}
\end{theorem}

The proof of Theorem~\ref{thm:ortho-approx} is given in Appendix~\ref{sec:proof_of_so3_var_approx}. It is worth mentioning that both $f_1$ and $f_2$ are monotone increasing with $f_1(0)=f_2(0)=0$ and $f_1(6)=f_2(6)=3$. They are in fact the lower and upper bound of $\tilde{L}^\orthgroup(\vg_1,\vg_2,\cdots,\vg_N)$ with given $L^\orthgroup(\vg_1,\vg_2,\cdots,\vg_N)$. Figure~\ref{fig:range} visualizes this approximation.

\section{The proposed numerical algorithm}

In this section, we present the semi-definite programming (SDP) relaxations of \eqref{formulation:rotation-approx} and \eqref{formulation:projection-approx} under the non-unique game (NUG) framework~\cite{Singer_2015_NUG} followed by a novel rounding algorithm for computing the representatives.

\subsection{The SDP relaxations}

Setting
\begin{equation*}
    f_{ij}(\vg) =
    \begin{cases}
        \frac{1}{2N^2}(\distort(\mR_i\vg,\mR_j))^2, & \quad\text{(spatial rotations)}, \\
        \frac{1}{2N^2}(\distsp(\vg^\top\vn_i,\vn_j))^2, & \quad\text{(projection directions)},
    \end{cases}
\end{equation*}
the minimization problems \eqref{formulation:rotation-approx} and \eqref{formulation:projection-approx} can be formulated as
\begin{equation}\label{min:NUG_original}
    \min_{\vg_1,\vg_2,\cdots,\vg_N\in\gG}~\sum_{1\leq i,j\leq N} f_{ij}(\vg_i\vg_j^{-1}).
\end{equation}
The problem \eqref{min:NUG_original} can be relaxed to a SDP, named the non-unique game (NUG) formulation~\cite{Singer_2015_NUG}, using the generalized Fourier transformation. The details are given in the next paragraph.

For a finite group $\gG$, there exists a list of group homomorphisms $\{\rho_k:\gG\to\mathrm{U}(d_k)\}_{k=0,\cdots,K-1}$ named unitary irreducible representations of $\gG$, where $\mathrm{U}(d_k) = \{\mX\in\sC^{d_k\times d_k}|\mX\mX^\HT = \mI_{d_k}\}$, $\mX^\HT$ is the conjugate transpose of $\mX$, and $d_k$ is the dimension of $\rho_k$. This work focuses on molecular symmetry groups, and their unitary irreducible representations are given in Section 3 of the supplementary material. We set $\rho_0$ to be trivial, i.e., $\rho_0(\vg)=1$ for all $\vg\in\gG$. For any function $f\in L_2(\gG)$ where $L_2(\gG)=\{f:\gG\mapsto\mathbb{C}|\int |f(g)|^2d\mu<\infty\}$ and $d\mu$ indicates integration with respect to the Haar measure on the group\footnote{In our case, $\gG$ is finite, and any complex valued function on $\gG$ belongs to $L_2(\gG)$.}, its generalized forward Fourier transform $\hat{f}$ is
\begin{equation*}
    \hat{f}(k) = \frac{1}{|\gG|}\sum_{\vg\in\gG} f(\vg)\rho_k(\vg)^\HT,\quad k=0,1,\cdots,K-1,
\end{equation*}
and the generalized inverse Fourier transform~\cite{Serre_1977_Representation_Theory} is
\begin{equation}\label{eq:generalized_inverse_Fourier_transform}
    f(\vg) = \sum_{k=0}^{K-1} d_k\Tr(\hat{f}(k)\rho_k(\vg)),
\end{equation}
where $\Tr(\mX)$ is the trace of $\mX$. Utilizing the above generalized forward/inverse Fourier transform, \eqref{min:NUG_original} is equivalent to the matrix form
\begin{equation}
    \begin{aligned}
        \min_{\vg_1,\cdots,\vg_N\in\gG} & \sum_{k=0}^{K-1} \Tr(\mF_k\mX^k), \\
        \text{s.t.} & \quad \mX_{ij}^k = \rho_k(\vg_i\vg_j^{-1}), \quad\forall 1\leq i,j\leq N, 0\leq k\leq K-1,
    \end{aligned}
\end{equation}
where $\mF_k=d_k(\hat{f}_{ij}(k))_{1\leq i,j\leq N}\in\sC^{Nd_k\times Nd_k}$. Since the group homomorphism $\rho_k$ satisfies $\rho_k(\vg_i\vg_j^{-1}) = \rho_k(\vg_i)\rho_k(\vg_j)^\HT\in\mathbb{C}^{d_k\times d_k}$, the matrix $\mX^k$ has the form:
\begin{equation}\label{equivalent-NUG}
    \mX^k =
        \begin{bmatrix}\rho_k(\vg_1) \\ \rho_k(\vg_2) \\ \vdots \\ \rho_k(\vg_N)\end{bmatrix}
        \begin{bmatrix}\rho_k(\vg_1) \\ \rho_k(\vg_2) \\ \vdots \\ \rho_k(\vg_N)\end{bmatrix}^\HT
        \in\mathbb{C}^{Nd_k\times Nd_k}.
\end{equation}
The NUG approach~\cite{Singer_2015_NUG} relaxes the constraints \eqref{equivalent-NUG} to
\begin{subequations}
\begin{align}
    & \mX^k\succeq 0,  \quad\forall 0\leq k\leq K-1, \label{con:a}\\
    & \mX_{ij}^0 = 1,  \quad\forall 1\leq i,j\leq N,\label{summation}\\
    & \mX_{ii}^k = \mI_{d_k},  \quad \forall 0\leq k\leq K-1, 1\leq i\leq N,\\
    & \sum_{k=0}^{K-1} d_k\Tr(\mX_{ij}^k\rho_k(\vg)^\HT) \geq0,  \quad \forall \vg\in \gG, 1\leq i,j\leq N.\label{convex-hull}
\end{align}
\end{subequations}
In summary, instead of solving \eqref{equivalent-NUG}, the NUG approach solves the problem:
\begin{equation}\label{NUG-sdp}
    \min_{\mX^0,\cdots,\mX^{K-1}} \sum_{k=0}^{K-1} \Tr(\mF_k\mX^k),\quad\text{s.t.}\quad \{\mX^k\}\text{ satisfies \eqref{con:a}-\eqref{convex-hull}},
\end{equation}
which is an SDP that can be solved by existing convex optimization solvers such as CVX \cite{Andersen_2013_CVXOPT}, SDPT3 \cite{Toh_2003_SDPT3}, SDPNAL \cite{Yang_2015_SDPNAL}. In \eqref{NUG-sdp}, there are $O(N^2|\gG|)$ variables and $O(N^2|\gG|)$ constraints in total, leading to a large-scale problem if $N$ and $|\gG|$ are huge. Therefore, designing an efficient solver for large-scale SDP is desirable. We will leave it as our future work.

\subsection{A rounding algorithm}\label{sec:rounding_algorithm}

Let $\{\mX^k\}_{k=0}^{K-1}$ to be the solution from \eqref{NUG-sdp}, it needs to design a rounding procedure for finding the representatives $\{\vg_i\}_{i=1}^N$. The rounding procedure in the original NUG framework~\cite{Singer_2015_NUG} is mainly designed for $\mathrm{SO}(2)$ or $\orthgroup$ and based on the eigenvalue decomposition of $\mX^1$. More specifically, let $\vv_1,\cdots,\vv_{d_1}$ to be the top $d_1$ eigenvectors of $\mX^1$, it finds $\{\vg_i\}_{i=1}^N$ using the approximation
\[\mV = \begin{bmatrix}\vv_1 & \vv_2 & \cdots & \vv_{d_1}\end{bmatrix} \approx \begin{bmatrix}\rho_1(\vg_1) \\ \rho_1(\vg_2) \\ \vdots \\ \rho_1(\vg_N)\end{bmatrix}.\]
It is worth mentioning that the consecutive $d_1$ rows in $\mV$ may not lie in the image of $\rho_1$, and a projection process is needed. The above method implicitly requires the injectiveness of $\rho_1$, which may not hold for other types of groups. Furthermore, the information in $\{\mX^k\}_{k=1}^{K-1}$ is not fully explored. Here, we propose a greedy algorithm for determining $\{\vg_i\}_{i=1}^N$ based on $\{\mX^k\}_{k=0}^{K-1}$ from \eqref{NUG-sdp}.

Let $[N] = \{1,2,\cdots,N\}$, we define a ``partial solution" function $s:[N]\times[N]\to\gG$ together with an indicator function $\hat{s}:[N]\times[N]\to\{0,1\}$ such that
\begin{equation*}
    \vg_i\vg_j^{-1}
    \begin{cases}
        = s(i,j), & \text{ if } \hat{s}(i,j)=1,\\
        \text{is undetermined}, & \text{ if } \hat{s}(i,j)=0.
    \end{cases}
\end{equation*}
and its ``partial cost" function
\begin{equation}
    \mathrm{PC}(s) = \sum_{\substack{i,j\in[N] \\ \hat{s}(i,j)=1}}f_{ij}(s(i,j)).
\end{equation}
Let $\ve$ be the identity element in $\gG$, we initialize $s(i,i)=\ve,\hat{s}(i,i)=1$ for any $i\in[N]$ and $\hat{s}(i,j)=0$ for $i\neq j$. Our goal is to update $s$ such that $\hat{s}(i,j)=1$ for any $1\leq i,j\leq N$, and $s$ must be a ``compatible'' partial solution, i.e., it satisfies the following properties:
\begin{itemize}
    \item $\forall i,j\in[N]$, if $\hat{s}(i,j)=1$ then $\hat{s}(j,i)=1$ and $s(j,i)=s(i,j)^{-1}$.
    \item $\forall i_1,i_2,\cdots,i_l\in[N]$, if $\hat{s}(i_1,i_2)=\hat{s}(i_2,i_3)=\cdots=\hat{s}(i_{l-1},i_l)=1$ then $\hat{s}(i_1,i_l)=1$ and $s(i_1,i_l)=s(i_1,i_2)\cdots s(i_{l-1},i_l)$.
\end{itemize}

To achieve the above goal, we update the ``partial solution" $s$ in a sequential way that consists of three steps:
\begin{itemize}
    \item {\bf Step 1:} Choose one index pair $(i,j)$ with $\hat{s}(i,j)=0$;
    \item {\bf Step 2:} Update the value of $s(i,j)$ such that $\hat{s}(i,j)=0$;
    \item {\bf Step 3:} Find the ``closure'' of $s$ via Algorithm~\ref{alg:TC-s}.
\end{itemize}

\begin{algorithm}[ht!]
    \caption{Cl($s$): the closure of $s$}
    \label{alg:TC-s}
    \begin{algorithmic}[1]
        \State Initialize $\mathrm{Cl}(s) \gets s$
        \Repeat
            \State Find $(i,j)\in[N]^2$ such that $\widehat{\mathrm{Cl}(s)}(i,j)=1$ but $\widehat{\mathrm{Cl}(s)}(j,i)=0$.
            \State Update $\mathrm{Cl}(s)$ by $\widehat{\mathrm{Cl}(s)}(j,i) \gets 1,\mathrm{Cl}(s)(j,i) \gets \mathrm{Cl}(s)(i,j)^{-1}$.
            \State Find $(o,p,q)\in[N]^3$ such that $\widehat{\mathrm{Cl}(s)}(o,p)=\widehat{\mathrm{Cl}(s)}(p,q)=1$ but $\widehat{\mathrm{Cl}(s)}(o,q)=0$.
            \State Update $\mathrm{Cl}(s)$ by $\widehat{\mathrm{Cl}(s)}(o,q) \gets 1,\mathrm{Cl}(s)(o,q) \gets \mathrm{Cl}(s)(o,p)\mathrm{Cl}(s)(p,q)$.
        \Until{there is no possible update of $\mathrm{Cl}(s)$.}
    \end{algorithmic}
\end{algorithm}

Now, we give the criterion for choosing the index pair $(i,j)$ in {\bf Step 1}. Define
\begin{equation}\label{coeff}
    \lambda_{ij}(\vg) = \frac{1}{|\gG|} \sum_{k=0}^{K-1} d_k\Tr(\mX_{ij}^k\rho_k(\vg)^\HT),
\end{equation}
the condition \eqref{convex-hull} is equivalent to $\lambda_{ij}(\vg) \geq 0$. Moreover,
\begin{equation}\label{coeff_sum}
    \begin{aligned}
        \sum_{\vg\in \gG}\lambda_{ij}(\vg)
        & = \sum_{k=0}^{K-1}d_k\Tr\left(\mX_{ij}^k\frac{1}{|\gG|}\sum_{\vg\in \gG}\rho_k(\vg)^\HT\right)  = d_0\Tr(\mX_{ij}^0) = 1,
    \end{aligned}
\end{equation}
where the second equality is from the Schur orthogonality relations (\textbf{Corollary 2} and \textbf{Corollary 3} in \textbf{Section 2.2}~\cite{Serre_1977_Representation_Theory}):
\begin{equation}
    \frac{1}{|\gG|}\sum_{\vg\in \gG}\rho_k(\vg)^\HT=\begin{cases}1 & k=0, \\ \vzero_{d_k} & k\neq 0,\end{cases}
\end{equation}
and the last equality is from \eqref{summation}. In addition,
\begin{equation}\label{eqn:convex-com}
    \mX_{ij}^k=\sum_{\vg\in \gG}\lambda_{ij}(\vg)\rho_k(\vg)
\end{equation}
holds by inverse Fourier transform. Combining \eqref{coeff}, \eqref{coeff_sum} and \eqref{eqn:convex-com}, it suggests that $\mX_{ij}^k$ is a convex combination of $\rho_k(\vg)$. Since~$\mX_{ij}^k = \rho_k(\vg_i \vg_j^{-1}) = \rho_k(\tilde{\vg})$ for some $\tilde{\vg}\in \gG$, \eqref{eqn:convex-com} is a natural relaxation of the constraint in \eqref{equivalent-NUG} and $\lambda_{ij}(\vg)$ indicates the ``probability'' of the event $\vg_i\vg_j^{-1} = \vg$. Suppose $M=|\gG|$ and $\gG=\{\vg^1,\vg^2,\cdots,\vg^M\}$, for each pair $(i,j)$, the group elements are sorted by the descending order of $\lambda_{ij}(\vg)$, i.e., find $\gG_{ij}=\{\vg_{ij}^1,\cdots,\vg_{ij}^M\}=\gG$ such that
\begin{equation}\label{eqn:sort}
    \lambda_{ij}(\vg_{ij}^1)\geq\lambda_{ij}(\vg_{ij}^2)\geq\cdots\geq\lambda_{ij}(\vg_{ij}^M).
\end{equation}
We choose the index as
\begin{equation}\label{criterion:index}
    (i,j) \in \argmax\{\lambda_{ij}(\vg_{ij}^1)|\hat{s}(i,j)=0\},
\end{equation}
and update $s(i,j)$ as $\vg_{ij}^1$.

However, to avoid bad local minimums, we simultaneously maintain at most $m$ partial solution functions $\{s_1,s_2,\cdots,s_l\}$, where $1\leq l\leq m$. Selecting the index $(i,j)$ as \eqref{criterion:index}, we find $L$ such that
\begin{equation}\label{eqn:L}
    L = \min\left\{k\left|\sum_{t=1}^k\lambda_{ij}(\vg_{ij}^t)\geq c\right.\right\},
\end{equation}
where $c\in[0,1]$ is a fixed threshold hyperparameter. Define candidate partial solution functions $r_{p,q}~(1\leq p\leq l, 1\leq q\leq L)$ by adding new guess of $\vg_i\vg_j^{-1}$ into $s_p$ as
\begin{equation}\label{eqn:candidate}
\begin{aligned}
    r_{p,q}(i',j') &=
    \begin{cases}
        \vg_{ij}^q, & \text{ if } (i',j') = (i,j); \\
        s_p(i',j'), & \text{ otherwise}.
    \end{cases} \\
    \hat{r}_{p,q}(i',j') &=
    \begin{cases}
        1, & \text{ if } (i',j')=(i,j); \\
        \hat{s}_p(i',j'), & \text{ otherwise}.
    \end{cases}
\end{aligned}
\end{equation}
Taking the closure of $r_{p,q}$ as $\mathrm{Cl}(r_{p,q})$, we then arrange them in the ascending order by the partial cost, i.e.,
\begin{equation}
    \mathrm{PC}(\mathrm{Cl}(r_{p^1,q^1}))\leq \mathrm{PC}(\mathrm{Cl}(r_{p^2,q^2}))\leq \cdots\leq \mathrm{PC}(\mathrm{Cl}(r_{p^{lL},q^{lL}})).
\end{equation}
Finally, we set $l\leftarrow\min\{lL,m\}$ and update the list $s_1,\cdots,s_l$ as $s_t\leftarrow\mathrm{Cl}(r_{p^t,q^t})$. The detailed rounding procedure is given in Algorithm~\ref{algo:greedy}.

\begin{algorithm}[ht]
    \caption{A greedy rounding algorithm}
    \label{algo:greedy}
    \begin{algorithmic}[1] 
        \State Input: two hyperparameters $m\in\mathbb{Z}^+,c\in[0,1]$.
        \State Initialization: $l=1$, $s_1(i,i) \gets \ve$, $\hat{s}_1(i,j) \gets \begin{cases} 1 & i=j \\ 0 & i\neq j\end{cases}$.
        \State Sort all group elements for each $(i,j)$ such that \eqref{eqn:sort} holds.
        \While{$\{(i,j)|\hat{s}_1(i,j)=0\}\neq\emptyset$}
            \State $(i,j) \gets \argmax\{\lambda_{ij}(\vg_{ij}^1)|\hat{s}_1(i,j)=0\}$.
            \State Compute $L$ as \eqref{eqn:L}.
            \State Compute candidate partial solutions $r_{p,q}$s as \eqref{eqn:candidate}.
            \State Compute $\mathrm{Cl}(r_{p,q}), 1\leq p\leq l, 1\leq q\leq L$ via Algorithm~\eqref{alg:TC-s}.
            \State Compute the partial costs of $\mathrm{Cl}(r_{p,g})$ and arrange them in the ascending order.
            \State Set $l$ as $ \min\{lL,m\}$.
            \State Choose $\{s_1,\cdots,s_l\} \gets $ top $l$ candidate $\mathrm{Cl}(r_{p^i,q^i})$.
        \EndWhile
        \State Output: $\vg_i=s_1(i,1)$.
    \end{algorithmic}
\end{algorithm}

The partial solutions and their closures in Algorithm~\ref{algo:greedy} can be efficiently maintained by a union-find data structure \cite{Cormen_2009_Introduction_to_Algorithms} in implementation, and priority queue data structure \cite{Cormen_2009_Introduction_to_Algorithms} is suitable for the list of candidates. The time complexity of sorting group elements for each $(i,j)$ is $O(M\log M)$, and then sorting the indices is $O(N^2\log N)$, in total $O(N^2M\log M+N^2\log N)$. The enumeration of the index $(i,j)$ takes $O(N^2)$. Computing partial costs needs $O(N^2)$ time in total. There are exactly $N-1$ executions of Line 6, each with $O(L)=O(M)$ time. Finally, the complexity of the query of $s(i,j)$ is negligible since it is $O(\alpha(N))$ where $\alpha$ is the inverse function of Ackermann function, which grows exceptionally slowly \cite{Cormen_2009_Introduction_to_Algorithms}. Therefore, given fixed $m$ and $c$, the time complexity of the while-loop is negligible compared to the sorting step (including Line 5). In conclusion, the total time complexity of our rounding algorithm is $O(N^2M\log M+N^2\log N)$, which is not a bottleneck of the whole algorithm comparing to solving \eqref{NUG-sdp}.

\section{Numerical experiments}

In this section, we first test the performance of the proposed approach on simulated data, measuring its capability of reaching global optimal. Then, we demonstrate the application of our method on a clustering problem using real-world cryo-EM data.

\subsection{Simulation experiments}\label{sec:simulation_experiments}

We validate our approach through the following two experiments:
\begin{enumerate}
    \item the approximation ability of $\tilde{L}^\orthgroup$ to $L^\orthgroup$;
    \item the global solution analysis of the NUG approach for solving $\tilde{L}^\orthgroup$ and $\tilde{L}^\sph$, and the comparison with the rounding technique in \cite{Singer_2015_NUG}.
\end{enumerate}

The simulation dataset contains 1000 cases. Each case randomly generates $N=1+\lceil\log_{|\gG|}1000\rceil$ points from $\orthgroup$ or $\sph$. Since we have to test all the possible combinations of $(\vg_1,\cdots,\vg_n)\in\gG^N$ to find the global minimum, the number of points $N$ can not be very large due to the limit of computational sources. We test the performance by considering the symmetry groups $\gG=\gC_2,\gC_7,\gD_2,\gD_7,\gT,\gO,\gI$ (see Section 2 of supplementary material and also ~\cite{Hu_2019}). The hyperparameters $m$ and $c$ are set to default value $20$ and $0.99$, respectively.

\subsubsection{The approximation ability of $\tilde{L}^\orthgroup$ to $L^\orthgroup$}

We use the arithmetic distance and compute the global minimum of $L^\orthgroup$ and $\tilde{L}^\orthgroup$ by the brute-force method, denoted as $\{\vg_i^{GL}\}_{i=1}^N$ and $\{\tilde{\vg}_i^{GL}\}_{i=1}^N$ respectively. We consider the following metrics:
\begin{align}
    & \text{Ratio of Equality (RoE)} = \frac{\text{ the number of trials with $\{\vg_i^{GL} = \tilde{\vg}_i^{GL}, i\in[N]\}$ }}{\text{the total number of trials}}, \label{metric:RoE}\\
    & \text{Relative Cost Gap (RCG)} = \frac{|L^\orthgroup(\{\vg_i^{GL}\}_{i=1}^N)-L^\orthgroup(\{\tilde{\vg}_i^{GL}\}_{i=1}^N)|}{|L^\orthgroup(\{\vg_i^{GL}\}_{i=1}^N)|}, \label{metric:RCG} \\
    & \text{Ratio of RCG}<p = \frac{\text{ the number of trials with $\text{RCG}<p$ }}{\text{the total number of trials}}, \quad p\in(0,1). \label{metric:RRCG}
\end{align}
The detailed results are given in Table~\ref{tab:performance_of_approximation}. The result shows that $\tilde{L}^\orthgroup$ can well approximate $L^\orthgroup$ for all symmetry groups except the cyclic groups. However, it is noted that there is a high probability that the relative cost gap is less than $0.1$, which validates the rationality of our approximation.

\begin{table}[!htbp]
    \centering
    \caption{The approximation ability of $\tilde{L}^\orthgroup$ to $L^\orthgroup$. RoE is defined in \eqref{metric:RoE} and ratio of RCG$<p$ is defined in \eqref{metric:RRCG}.}
    \begin{tabular}{c|ccc}
        \toprule
        Group $\gG$ & RoE & ratio of RCG$<0.01$ & ratio of RCG$<0.1$ \\
        \midrule
        $\gC_2$ & 33.7\% &  50.6\% &  95.3\% \\
        $\gC_7$ & 44.8\% &  60.5\% &  94.2\% \\
        $\gD_2$ & 93.4\% &  96.7\% & 100.0\% \\
        $\gD_7$ & 96.1\% &  99.1\% & 100.0\% \\
        $\gT$   & 98.4\% &  99.9\% & 100.0\% \\
        $\gO$   & 98.6\% & 100.0\% & 100.0\% \\
        $\gI$   & 99.9\% & 100.0\% & 100.0\% \\
        \bottomrule
    \end{tabular}
    \label{tab:performance_of_approximation}
\end{table}

\subsubsection{The NUG approach for solving $\tilde{L}^\orthgroup$ and $\tilde{L}^\sph$}

Same as the previous subsection, we calculate the global optimal value of $\tilde{L}^\orthgroup$ and $\tilde{L}^\sph$ by the brute-force method, denoted as $\tilde{L}^\orthgroup_{GL}$ and $\tilde{L}^\sph_{GL}$ respectively. Also, we calculate the solution by the SDP relaxations and the proposed rounding algorithm~\ref{algo:greedy}. We substitute the NUG solution into $\tilde{L}^\orthgroup$ and $\tilde{L}^\sph$, defined as $\mathrm{NUG}^{\orthgroup}$ and $\mathrm{NUG}^{\sph}$ respectively. Define the relative cost gap of the NUG approach (RCG-NUG) as
\begin{equation*}
    \text{RCG-NUG} = \begin{cases} \frac{|\tilde{L}^\orthgroup_{GL}-\mathrm{NUG}^\orthgroup|}{|\tilde{L}^\orthgroup_{GL}|}, & \text{ (Spatial rotations)}; \\
    \frac{|\tilde{L}^\sph_{GL}-\mathrm{NUG}^\sph|}{|\tilde{L}^\sph_{GL}|}, & \text{ (Projection directions)}.
    \end{cases}
\end{equation*}
We evaluate the performance of the proposed method by the following two criteria:
\begin{equation}\label{metric:NUG}
    \begin{aligned}
        \text{Accuracy} & = \frac{\text{the number of trials with RCG-NUG=0}}{\text{the total number of trials}}, \\
        \text{Maximal RCG-NUG} & = \text{the maximal value of RCG-NUG among all trials}.
    \end{aligned}
\end{equation}
Notice that by $\text{RCG-NUG}=0$ we mean the NUG solution of $\{\vg_i^{GL}\}_{i=1}^N$ is identical to $\{\tilde{\vg}_i^{GL}\}_{i=1}^N$. The detailed results are given in Table~\ref{tab:performance_of_the_NUG_approach}. Note that we test the performance of NUG approach for both arithmetic distance $d_\orthgroup^A,d_\sph^A$ and geometric distance $d_\orthgroup^G,d_\sph^G$. Our method almost recovers all the global solutions under different molecular symmetries and distances on $\orthgroup$ and $\sph$. The relative cost gap is slim for the trials in which the NUG approach fails to obtain the global minimum.

\begin{table}[!htbp]
    \centering
    \caption{The global optimal analysis of the NUG approach. (Accuracy, Maximal RCG-NUG) are defined in \eqref{metric:NUG}.}
    \begin{tabular}{c|cc|cc}
        \toprule
        \multirow{2}*{Group $\gG$} &
        \multicolumn{2}{c|}{Spatial rotations} &
        \multicolumn{2}{c}{Projection directions} \\
        & $d_\orthgroup^A$ & $d_\orthgroup^G$ & $d_\sph^A$ & $d_\sph^G$ \\
        \midrule
        $\gC_2$ & (98.1\%, 0.82\%) & (98.7\%, 1.32\%) &  (100\%,    0\%) & (100\%, 0\%) \\
        $\gC_7$ & (99.9\%, 0.22\%) & (98.7\%, 2.33\%) &  (100\%,    0\%) & (100\%, 0\%) \\
        $\gD_2$ & (99.9\%, 0.13\%) &  (100\%,    0\%) & (99.9\%, 0.02\%) & (100\%, 0\%) \\
        $\gD_7$ &  (100\%,    0\%) & (99.9\%, 0.58\%) &  (100\%,    0\%) & (100\%, 0\%) \\
        $\gT$   &  (100\%,    0\%) &  (100\%,    0\%) &  (100\%,    0\%) & (100\%, 0\%) \\
        $\gO$   &  (100\%,    0\%) &  (100\%,    0\%) &  (100\%,    0\%) & (100\%, 0\%) \\
        $\gI$   &  (100\%,    0\%) &  (100\%,    0\%) &  (100\%,    0\%) & (100\%, 0\%) \\
        \bottomrule
    \end{tabular}
    \label{tab:performance_of_the_NUG_approach}
\end{table}

Also, we test the performance of the rounding algorithm proposed in the original NUG method~\cite{Singer_2015_NUG} to compare it with the proposed algorithm. As we discussed in \textbf{Section} \ref{sec:rounding_algorithm}, we only implement it when $\gG$ is cyclic since $\rho_1$ of other molecular symmetry groups is not injective. The results are given in Table~\ref{tab:comparison_with_Singer_s_rounding}. Both accuracy and maximal RCG-NUG are noticeably lower than the proposed algorithm.

\begin{table}[!htbp]
    \centering
    \caption{The (Accuracy, Maximal RCG-NUG) results of by replacing the rounding algorithm in the NUG approach with that proposed~\cite{Singer_2015_NUG}.}
    {\small \begin{tabular}{c|cc|cc}
        \toprule
        \multirow{2}*{Group $\gG$} & \multicolumn{2}{c|}{Spatial rotations} &
        \multicolumn{2}{c}{Projection directions} \\
        & $d_\orthgroup^A$ & $d_\orthgroup^G$ & $d_\sph^A$ & $d_\sph^G$ \\
        \midrule
        $\gC_2$ & (81.1\%, 4.32\%) & (80.4\%,  7.88\%) & (90.1\%, 7.94\%) & (89.2\%, 11.76\%) \\
        $\gC_7$ & (92.2\%, 6.66\%) & (84.9\%, 10.06\%) & (99.0\%, 4.21\%) & (98.9\%,  1.66\%) \\
        \bottomrule
    \end{tabular}}
    \label{tab:comparison_with_Singer_s_rounding}
\end{table}

Next, we test the sensitivity of our method (see details in Section 5 of the supplementary material), varying the number of ``partial solutions" $m$ and the threshold probability $c$ under the distance $d_\orthgroup^A$ (Table~\ref{tab:hyperparameters}). $m$ can be reduced to 12 without sacrificing the accuracy if $c=0.99$, compared to the the first column ($m=20,c=0.99$) in Table~\ref{tab:performance_of_the_NUG_approach}. Meanwhile, the threshold probability $c$ significantly affects the accuracy, making it a vital parameter in the proposed greedy algorithm. $m = 20, c = 0.99$ is used as the default setting in real-world applications, including the demonstration of a real-world cryo-EM clustering problem in the following paragraphs.

\begin{table}[!htbp]
    \centering
    \caption{The (Accuracy, Maximal RCG-NUG) results for varying $m$ and $c$ under the $d_\orthgroup^A$. $c=0$ means $L=1$ in Algorithm~\ref{algo:greedy}.}
    {\small \begin{tabular}{c|cc|cc}
        \toprule
        \multirow{2}*{Group $\gG$} & \multicolumn{2}{c|}{$c=0.99$} & \multicolumn{2}{c}{$m=20$} \\
        & $m=12$ & $m=4$ & $c=0.5$ & $c=0$ \\
        \midrule
        $\gC_2$ & (98.1\%, 0.82\%) & (98.0\%, 0.82\%) & (89.6\%,  2.40\%) & (89.6\%,  2.40\%) \\
        $\gC_7$ & (99.8\%, 0.50\%) & (97.9\%, 8.27\%) & (95.7\%,  2.48\%) & (95.3\%,  2.48\%) \\
        $\gD_2$ & (99.9\%, 0.13\%) & (99.9\%, 0.13\%) & (94.6\%,  8.74\%) & (94.6\%,  8.74\%) \\
        $\gD_7$ &  (100\%, 0\%)    & (99.8\%, 1.62\%) & (96.8\%, 13.11\%) & (96.5\%, 17.62\%) \\
        $\gT$   &  (100\%, 0\%)    &  (100\%, 0\%)    & (98.2\%, 33.32\%) & (98.2\%, 33.32\%) \\
        $\gO$   &  (100\%, 0\%)    &  (100\%, 0\%)    & (99.2\%,  6.89\%) & (99.2\%,  6.89\%) \\
        $\gI$   &  (100\%, 0\%)    &  (100\%, 0\%)    & (99.7\%,  2.00\%) & (99.7\%,  2.00\%) \\
        \bottomrule
    \end{tabular}}
    \label{tab:hyperparameters}
\end{table}

Finally, we measure the running time of our algorithm in its default setting. The configuration of this experiment is identical to the previous one, except that we fix the problem size $N=10$ and $20$ since we do not need to employ brute force enumeration to obtain the ground truth solution. We test our algorithm on a modern CPU computation node, which is equipped with 2 Intel(R) Xeon(R) Gold 6230 CPUs @ 2.10GHz totaling 40 cores and 188GB of RAM. In each test case, we measure the average running time of one mean and variance computation. The results are presented in Table~\ref{tab:running_time}. Thus, it is desirable to find an efficient and scalable minimization algorithm to solve \eqref{NUG-sdp}, which will be our next goal.

\begin{table}[!htbp]
    \centering
    \caption{The running time of one mean and variance computation by our algorithm. For $\gG=\gI$ and $N=20$, the scale of optimization problem is too large to be solved.}
    \begin{tabular}{c|cc|cc}
        \toprule
        \multirow{2}*{Group $\gG$} & \multicolumn{2}{c|}{Spatial rotations} & \multicolumn{2}{c}{Projection directions} \\
        & $N=10$ & $N=20$ & $N=10$ & $N=20$ \\
        \midrule
        $\gC_2$ &  0.27s &   0.96s &  0.25s &   0.89s \\
        $\gC_7$ & 14.07s & 253.29s & 10.59s & 106.43s \\
        $\gD_2$ &  0.94s &   4.46s &  0.93s &   4.85s \\
        $\gD_7$ &  5.55s &  46.09s &  5.05s &  46.44s \\
        $\gT$   & 20.50s & 173.66s & 17.92s & 175.37s \\
        $\gO$   & 13.00s & 172.99s & 12.46s & 157.39s \\
        $\gI$   & 63.36s &   NA    & 85.07s &   NA    \\
        \bottomrule
    \end{tabular}
    \label{tab:running_time}
\end{table}

\subsection{K-means clustering under molecular symmetry}

The K-means clustering algorithm is a classical clustering algorithm that consists of two key steps: 1. calculate the mean of a clustering result; 2. assign the label of each point using the updated mean of each cluster. This section considers the clustering problem under the $\gC_3$ molecular symmetry group. Here, we replace the distance with the arithmetic distance $d_\sph^A$ and calculate the mean using the proposed method by setting $m=20$ and $c=0.99$. As an additional note, the geometric distance can also be used, and the experimental results are identical.

Projection directions containing five classes are generated. Each class contains $100$ projection directions, evenly distributed on the quotient circle, of which the radius is $0.2$ and the center is $\left[\left(0, \frac{1}{2}, \frac{\sqrt{3}}{2}\right)^\top\right]$ (red), $\left[\left(\frac{3}{4}, \frac{\sqrt{3}}{4}, \frac{1}{2}\right)^\top\right]$ (yellow), $\left[\left(1, 0, 0\right)^\top\right]$ (cyan), $\left[\left(\frac{3}{4}, \frac{\sqrt{3}}{4}, -\frac{1}{2}\right)^\top\right]$ (lime) and $\left[\left(0, \frac{1}{2}, -\frac{\sqrt{3}}{2}\right)^\top\right]$ (pink), respectively, visualized in Figure~\ref{fig:Projection_clustering}a.

We apply the proposed algorithm for dividing these points into five clusters, resulting in $100\%$ clustering accuracy Figure~\ref{fig:Projection_clustering}b. As a comparison, we conduct the conventional K-means method on the same datasets. The conventional K-means method selects a fundamental domain of $\sph/\gG$, and then rotates each data point into the fundamental domain. Finally, it uses the distance $d_\sph^A$ on $\sph$ to cluster the transformed points with the classical K-means algorithm. The result obtained by the conventional method varies according to random seeds. Clustering in four runs is plotted in Figure~\ref{fig:Projection_clustering}c, of which clustering accuracy varies from $60.0\%$ to $75.6\%$. In contrast to our proposed clustering method, the actual topology induced by molecular symmetry, or say, true distance in quotient space, is neglected. The closer the projection direction is to the edge of the fundamental domain, the more severe the consequences (i.e., error) of neglecting the molecular symmetry.

\begin{figure}[htbp]
    \centering
    \includegraphics[width=\textwidth]{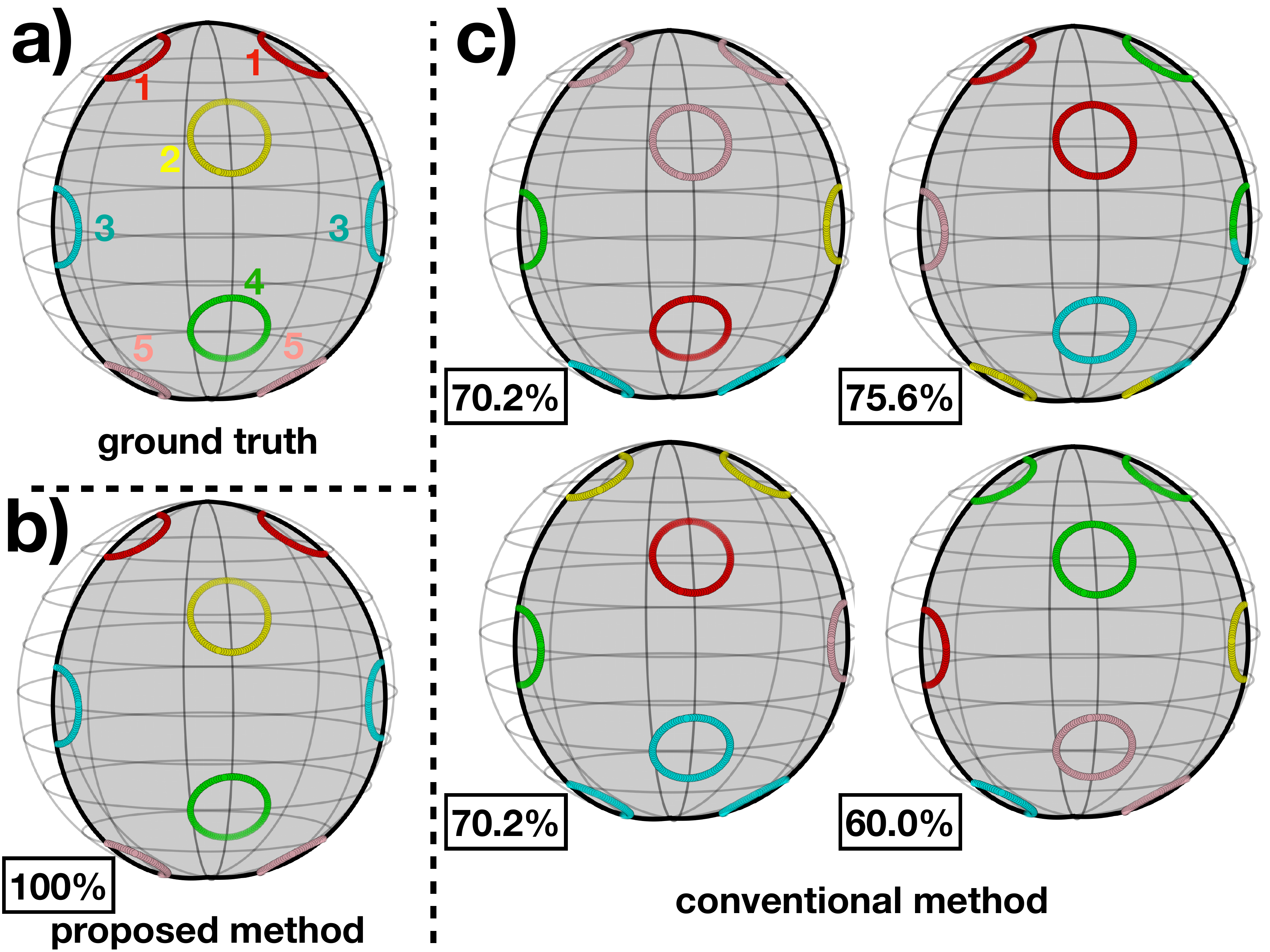}
    \caption{\textbf{Clustering of projection directions considering $\gC_3$ molecular symmetry group.} Five colors distinguish the five clusters. For presentation purposes, in each $[\vn_i]$, an element within a given fundamental domain is selected for display. The spheres are rotated in order to place projection directions toward the reader. \textbf{a}, Projection directions contain five classes. Each class includes $100$ projection directions, evenly distributed on the quotient circle, of which the radius is $0.2$ and the center is $\left[\left(0, \frac{1}{2}, \frac{\sqrt{3}}{2}\right)^\top\right]$ (red), $\left[\left(\frac{3}{4}, \frac{\sqrt{3}}{4}, \frac{1}{2}\right)^\top\right]$ (yellow), $\left[\left(1, 0, 0\right)^\top\right]$ (cyan), $\left[\left(\frac{3}{4}, \frac{\sqrt{3}}{4}, -\frac{1}{2}\right)^\top\right]$ (lime) and $\left[\left(0, \frac{1}{2}, -\frac{\sqrt{3}}{2}\right)^\top\right]$ (pink). \textbf{b and c}, $10$ projection directions are randomly selected from each cluster to compute the mean during K-means clustering. Clustering accuracy is labeled on the left lower corner. \textbf{b}, K-means clustering by the proposed method. \textbf{c}, K-means clustering by the conventional method. Four runs are demonstrated.}
    \label{fig:Projection_clustering}
\end{figure}

\subsection{Clustering projection directions to 2D asymmetric feature visualization in cryo-EM}
\label{subsec:assymetric-visulization}

In this section, we applied the proposed algorithm to cryo-EM data and demonstrated its effectiveness in addressing the symmetry mismatch issue by using a 2D asymmetric visualization method.

We used a simulated dataset of an icosahedral virus that possesses an asymmetric feature. A total of $100,000$ single-particle images were generated using RELION's \texttt{relion\_project} submodule~\cite{Scheres_2012} from the density map of the Q$\beta$-MurA complex~\cite{cui2017structures}, of which EMDB~\cite{iudin2023empiar} entry is EMD-8711. Within the Q$\beta$-MurA complex, the Q$\beta$ follows icosahedral symmetry, while the binding of MurA introduces the asymmetric feature (Figure~\ref{fig:Symmetry_Mismatch}a). Each image corresponds to a projection at a random orientation. White noise was introduced to the projection images, resulting in an SNR of -10dB (Figure~\ref{fig:Symmetry_Mismatch}a). We opted for $100,000$ images and an SNR of -10dB, as these values are representative of the typical order of magnitude observed in cryo-EM. Using the method proposed in the previous section, the simulated particle images were clustered according to their respective projection directions into $10$ clusters under the arithmetic distance, considering the $\gI$ molecular symmetry (Figure~\ref{fig:Symmetry_Mismatch}b). Cluster 1, comprising $9,850$ images, underwent 12 rounds of balanced K-means clustering using cosine similarity of particle images, resulting in $100$ clusters (Figure~\ref{fig:Symmetry_Mismatch}c). As clustering iterations progressed, the asymmetric feature (MurA) became increasingly pronounced (Figure~\ref{fig:Symmetry_Mismatch}c). Compared to RELION~\cite{Scheres_2012}, which is the mainstream image processing software in cryo-EM, where no asymmetry feature appears in its 2D clustering (often referred to as 2D classification in the cryo-EM field), our proposed method clearly demonstrates the ability to observe the asymmetry feature during the 2D clustering stage (Figure~\ref{fig:Symmetry_Mismatch}d).

\begin{figure}[htbp]
    \centering
    \includegraphics[width=\textwidth]{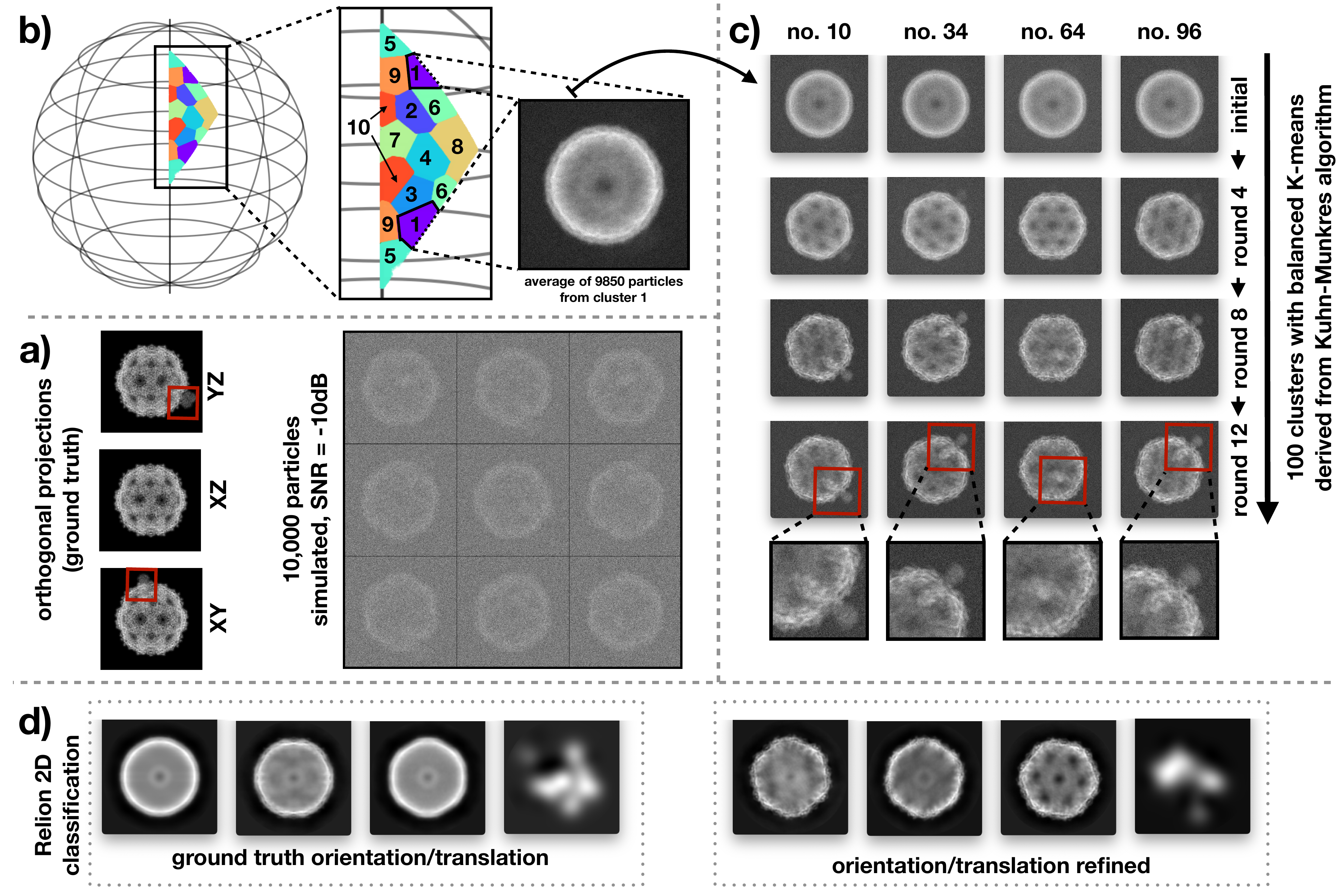}
    \caption{\textbf{Observing the asymmetry feature (MurA-binding) of the icosahedral Q$\beta$ viron via clustering of projection directions considering $\gI$ molecular symmetry group.} \textbf{a}, The density map of the Q$\beta$-MurA complex was taken from the EMDB (EMD-8711)~\cite{cui2017structures} as ground truth. Q$\beta$ follows icosahedral symmetry, while its MurA binding forms an asymmetric feature. Orthogonal projections of $YZ$, $XZ$, and $YZ$ are shown on the left, with MurA highlighted in red boxes. $100,000$ single particle images of the Q$\beta$-MurA complex were generated via RELION's \texttt{relion\_project} submodule. Each image corresponds to a projection at a random orientation. White noise was added to the projection images, making their SNR -10dB. $9$ synthetic single particle images are shown on the right. In cryo-EM, the number of images is of the order of magnitude of $10,000$, and each single particle image commonly exhibits an SNR of -10dB. \textbf{b}, Simulated single particle images were clustered by their corresponding projection directions into $10$ clusters when considering $\gI$ molecular symmetry. Different colors were used to distinguish clusters, with their respective identifiers marked. The images from cluster 1 (colored violet), which contains $9,850$ images, were selected, averaged, and then displayed. \textbf{c}, The $9,850$ images of cluster 1 underwent $12$ rounds of balanced K-means clustering into $100$ clusters, derived from the Kuhn-Munkres algorithm with image's cosine-similarity as the metric. The averages of four selected clusters out of the $100$ are depicted. The asymmetric feature (MurA), which becomes gradually more apparent as the clustering iterations proceed, is emphasized by red boxes and zoomed in for detail. \textbf{d}, The results of RELION clustering into $100$ clusters, commonly referred to as 2D classification in the field of cryo-EM and intended for comparison, were displayed, both without and with orientation refinement. For the tests without orientation refinement, ground truth orientations were imported into RELION, with no refinement of orientation performed. In contrast, for the tests with orientation refinement, RELION attempted to refine the orientation of each single particle image. In either case, averages from four out of the $100$ clusters were depicted. The first three represented the ``good'' clusters, which reflected the shape of the virus, while the last one represented the ``bad'' clusters, which constituted the majority of the $100$ categories.}
    \label{fig:Symmetry_Mismatch}
\end{figure}

\section{Conclusion and discussion}

In this work, we proposed an approximation algorithm for calculating the mean and the variance for spatial rotations and projection directions considering molecular symmetry. To solve this challenging optimization problem, we propose a relaxation method that can find the global minimum with a high probability according to numeric evidence. Finally, we demonstrate the application of our method for the generalized K-means clustering and 2D asymmetric feature visualization in cryo-EM. We release our method as an open-source Python package pySymStat (\url{https://github.com/mxhulab/pySymStat}).

Despite demonstrating good performance in experimental data, our approach does have several limitations. Firstly, the approximation of $\tilde{L}^{\SO{3}}$ to $L^{\SO{3}}$ is not sufficiently accurate, particularly for cyclic group $\gG_n$, as evidenced by the results presented in Table \ref{tab:performance_of_approximation}. Secondly, solving the SDP relaxation poses a computational bottleneck in our approach. Currently, the number of orientations cannot exceed 20. In order to obtain high-resolution asymmetric structures in cryo-EM with a large number of particles, it is crucial to develop a fast and efficient solver for handling the large-scale SDP relaxation problem outlined in equation \eqref{NUG-sdp}. This computational challenge is of utmost importance going forward. Thirdly, although our rounding process demonstrates a high rate of global optimality in our simulated experiments, a theoretical analysis is still lacking. Further research is needed to provide a rigorous theoretical understanding of the rounding procedure. Solving these problems can provide an effective solution to the symmetry mismatch issue, potentially increasing the resolution of the asymmetric unit attached to a symmetric unit.

For further discussion, we might further integrate the proposed method into some existing reconstruction algorithms in cryo-EM. In particular, the quaternion-assisted angular reconstruction algorithm~\cite{Farrow_1992, Farrow_1993} contributes to solving a series of structures at 10-30\AA\ resolution in the 90s~\cite{van_Heel_1997}. The critical step in this algorithm is solving the absolute orientation problem~\cite{Horn_1987, Harauz_1990}, which is related to determining the mean of a series of spatial rotations. In recent years, Shkolnisky and his colleagues applied an angular reconstruction algorithm into ab initio modeling in cryo-EM~\cite{Shkolnisky_2019,Shkolnisky_2020,geva2022common}, in the cases of cyclically, dihedrally, tetrahedrally and octahedrally symmetric molecules. However, such \textsl{ab initio} modeling remains unsolved for icosahedrally symmetric molecules. Therefore, it is possible to solve the aforementioned problems with the help of the proposed algorithm.

\section*{Acknowledgements}

The authors are grateful to Dr. Yichen Zhou for revising the manuscript.

\section*{Contributions}

M.H., Q.Z., C.B., and H.L. initialized the project and designed the methods. Q.Z., C.B., and H.L. clarified mathematical issues and derived formulas. M.H. and Q.Z. wrote pySymStat and validated its performance. M.H. and C.B. wrote the manuscript. The corresponding authors are CL Bao, H Lin, and MX Hu.

\bibliographystyle{plain}
\bibliography{main}

\appendix
\newpage

\section{The proof of Theorem \ref{thm:ortho-approx}}\label{sec:proof_of_so3_var_approx}

Write $L^\orthgroup(\vg_1,\cdots,\vg_N)$ and $\tilde{L}^\orthgroup(\vg_1,\cdots,\vg_N)$ as $L^\orthgroup$ and $\tilde{L}^\orthgroup$ for short, respectively. By direction computation,
\[ \begin{aligned}
    L^\orthgroup
    = & \min_{\mR\in\orthgroup}\frac{1}{N}\sum_{i=1}^N\|\mR-\mR_i\vg_i\|_F^2 \\
    = & \min_{\mR\in\orthgroup}\frac{1}{N}\sum_{i=1}^N(\|\mR\|_F^2-2\langle\mR,\mR_i\vg_i\rangle_F+\|\mR_i\vg_i\|_F^2) \\
    = & \min_{\mR\in\orthgroup}\|\mR\|_F^2-2\langle\mR,\tilde{\mR}\rangle_F+3 \\
    = & \min_{\mR\in\orthgroup}2(3-\langle\mR,\tilde{\mR}\rangle_F),
\end{aligned} \]
where $\tilde{\mR}=\frac{1}{N}\sum_{i=1}^N\mR_i\vg_i$. To find the minimizer $\bar{\mR}$ of this problem, it is equivalent to minimize $\|\bar{\mR}-\tilde{\mR}\|_F^2$, which is the well-known constrained orthogonal Procrustes problem. Let $\tilde{\mR}=\mU\bm{\Sigma}\mV^\top$ be the singular value decomposition of $\tilde{\mR}$, where $\mU,\mV\in\mathrm{O}(3)$, $\bm{\Sigma}=\operatorname{diag}(\sigma_1,\sigma_2,\sigma_3)$, $\sigma_1\geq\sigma_2\geq\sigma_3\geq0$. Then the Kabsch algorithm~\cite{Umeyama_1991} gives
\[ \bar{\mR} = \mU\operatorname{diag}(1,1,\epsilon)\mV^\top \]
where $\epsilon=\det(\mU\mV^\top)$ is the sign of $\mU\mV^\top\in\mathrm{O}(3)$. It follows that
\[ L^\orthgroup = 2(3-(\sigma_1+\sigma_2+\epsilon\sigma_3)).\]
On the other hand, notice that \eqref{variance:Euclidean} also holds for Frobenius norm of matrices, i.e.,
\[ \tilde{L}^\orthgroup = \frac{1}{N}\sum_{i=1}^N\|\tilde{\mR}-\mR_i\vg_i\|_F^2 = 3-\langle\tilde{\mR},\tilde{\mR}\rangle_F = 3-(\sigma_1^2+\sigma_2^2+\sigma_3^2). \]
The above arguments show that $L^\orthgroup$ and $\tilde{L}^\orthgroup$ are related to the singular value decomposition of $\tilde{\mR}=\frac{1}{N}\sum_{i=1}^N\mR_i\vg_i$. To characterize the relationship between $L^\orthgroup$ and $\tilde{L}^\orthgroup$, we need to vary the rotation matrices $\{\mR_{\vq_i}\}$ and the total number of points $N$ in $\tilde{\mR}$. Define the set
\begin{equation}
    \mathcal{R} = \left\{\tilde{\mR} = \frac{1}{N}\sum_{i=1}^N\mR_i\vg_i\left|\mR_i\in\orthgroup, i\in[N],  N\in\sN \right.\right\},
\end{equation}
we prove that $\mathcal{R}$ is a dense subset of $\Conv\orthgroup$, the convex hull of $\orthgroup$ in $\mathbb{R}^{3\times 3}$. On the one hand, it is simple that $\tilde{\mR}=\frac{1}{N}\sum_{i=1}^N\mR_i\vg_i \in \Conv\orthgroup$. On the other hand, for any $\mR\in\Conv\orthgroup$, suppose $\mR=\sum_{j=1}^T\mu_j\mR_j$, where $\mu_j\geq 0$ and $\sum_j\mu_j=1$. We can choose a large enough $N$ and a set of natural number $a_j$ such that $\frac{a_j}{N}\to\mu_j$. Then $\frac{1}{N}\sum_{j=1}^N\sum_{k=1}^{a_j}\mR_j \to \mR$.

Notice that $\Conv\orthgroup$ is invariant under the left/right multiplication of $\orthgroup$. Denoting $\mU'=\mU\operatorname{diag}(1,1,\det\mU),\mV'=\mV\operatorname{diag}(1,1,\det\mV)$, we have $\mU',\mV'\in\orthgroup$ and $\tilde{\mR}\in\Conv\orthgroup$ if and only if $\operatorname{diag}(\sigma_1,\sigma_2,\epsilon\sigma_3)=\mU'^\top\tilde{\mR}\mV\in\Conv\orthgroup$. Moreover, it is from from Proposition 4.1 in \cite{Sanyal_2011_Orbitopes} that the diagonal matrix ${\bf\Sigma}'=\operatorname{diag}(\sigma_1,\sigma_2,\epsilon\sigma_3)$ is in $\Conv\orthgroup$ if and only if the diagonal matrix
\[ \operatorname{diag}(
\sigma_1+\sigma_2+\epsilon\sigma_3+1,
\sigma_1-\sigma_2-\epsilon\sigma_3+1,
-\sigma_1+\sigma_2-\epsilon\sigma_3+1,
-\sigma_1-\sigma_2+\epsilon\sigma_3+1) \]
is positive semidefinite. The above condition is equivalent to
\[ \begin{cases}
    \sigma_1+\sigma_2+\epsilon\sigma_3+1 \geq 0, \\
    \sigma_1-\sigma_2-\epsilon\sigma_3+1 \geq 0,\\
    -\sigma_1+\sigma_2-\epsilon\sigma_3+1 \geq 0, \\
    -\sigma_1-\sigma_2+\epsilon\sigma_3+1 \geq 0.
\end{cases}. \]
By adding the last two inequalities, we get $\sigma_1\leq 1$. Combining it with the condition $\sigma_1\geq\sigma_2\geq\sigma_3\geq 0$ and $\epsilon=\pm 1$, the first three inequalities hold automatically. Therefore, $\tilde{\mR}\in\Conv\orthgroup$ if and only if
\begin{equation}\label{constraint-set}
    \sigma_1+\sigma_2-\epsilon\sigma_3\leq 1,\quad  1\geq\sigma_1\geq\sigma_2\geq\sigma_3\geq0,\quad\varepsilon = \pm 1.
\end{equation}
Finally, we consider the set
\begin{equation}\label{set:bound}
    \{ (L^\orthgroup=2(3-(\sigma_1+\sigma_2+\epsilon\sigma_3)), \tilde{L}^\orthgroup=3-(\sigma_1^2+\sigma_2^2+\sigma_3^2))\}\subset\sR^2
\end{equation}
subject to the constraints given in~\eqref{constraint-set}. After tedious but elementary calculations~(see details in Section 4 of the supplementary material), we know the lower bound and upper bound in Theorem~\ref{thm:ortho-approx} hold.

\section{Proof of well-definedness of the distance between two spatial rotations considering molecular symmetry}

	In this section, Equation (2.12) and (2.13) are proved to be well-defined distances, i.e., it is symmetric, positive definite, and satisfy the triangle inequality.

	As $\distort$ is symmetric,
	\[ \distortho([\mR_1],[\mR_2]) = \min_{\vg_1,\vg_2\in\gG}\distort(\mR_1\vg_1,\mR_2\vg_2) = \distortho([\mR_2],[\mR_1]) \]
	holds for all $\mR_1,\mR_2 \in \orthgroup$, which is the symmetry property.

	The positive definiteness contains two parts. The first part is that
	\[ \distortho([\mR_1],[\mR_2]) \geq 0 \]
	holds for all $\mR_1,\mR_2\in\orthgroup$, which is natural by the positive definiteness of $\distort$. The second part is that
	\[ \distortho([\mR_1],[\mR_2]) = 0 \]
	if and only if $[\mR_1]=[\mR_2]$, which is proven as follows. $[\mR_1]=[\mR_2]$ holds if and only if there exist $\vg_1,\vg_2\in\gG$ such that $\mR_1\vg_1=\mR_2\vg_2$. It is equivalent to $\distort(\mR_1\vg_1,\mR_2\vg_2) = 0$, since $\distort$ itself is positive definite. Therefore, it is further equivalent to $\distortho([\mR_1],[\mR_2]) = 0$, which completes the proof.

	The triangle inequality property is that
	\[ \distortho([\mR_1],[\mR_2]) \leq \distortho([\mR_1],[\mR_3]) + \distortho([\mR_3],[\mR_2]) \]
	holds for all $[\mR_1],[\mR_2],[\mR_3]\in\orthgroup$. It is proven as follows.
	\[ \begin{aligned}
		 & \distortho([\mR_1],[\mR_2]) \\
		= & \min_{\vg_1,\vg_2\in\gG} \distort(\mR_1\vg_1,\mR_2\vg_2) \\
		\leq & \min_{\vg_1,\vg_2,\vg_3\in\gG}\left(\distort(\mR_1\vg_1,\mR_3\vg_3) + \distort(\mR_3\vg_3,\mR_2\vg_2)\right)
	\end{aligned} \]
	because $\distort(\mR_1\vg_1,\mR_2\vg_2) \leq \distort(\mR_1\vg_1,\mR_3\vg_3) + \distort(\mR_3\vg_3,\mR_2\vg_2)$ holds for any $\vg_3 \in \gG$, as $\distort$ also has the triangle inequality property. Therefore,
	\[ \begin{aligned}
		& \distortho([\mR_1],[\mR_2]) \\
		\leq & \min_{\vg_1,\vg_2,\vg_3\in\gG}\left(\distort(\mR_1\vg_1,\mR_3\vg_3) + \distort(\mR_3\vg_3,\mR_2\vg_2)\right) \\
		= & \min_{\vg_1,\vg_2,\vg_3\in\gG}\left(\distort(\mR_1,\mR_3\vg_3\vg_1^{-1}) + \distort(\mR_3,\mR_2\vg_2\vg_3^{-1})\right) \\
		= & \min_{\vg_1',\vg_2'\in\gG}\left(\distort(\mR_1,\mR_3\vg_1') + \distort(\mR_3,\mR_2\vg_2')\right) \\
		= & \min_{\vg_1'\in\gG}\distort(\mR_1,\mR_3\vg_1') + \min_{\vg_2'\in\gG}\distort(\mR_3,\mR_2\vg_2) \\
		= & \min_{\vg_1,\vg_3\in\gG}\distort(\mR_1\vg_1,\mR_3\vg_3)+\min_{\vg_2,\vg_3\in\gG}\distort(\mR_3\vg_3,\mR_2\vg_2) \\
		= & \distortho([\mR_1],[\mR_3]) + \distortho([\mR_3],[\mR_2]).
	\end{aligned} \]

\section{Molecular symmetry groups}

As molecules are chiral in cryo-EM, there are only five classes of molecular symmetry groups. These are the cyclic group (\(\gC_n\)), dihedral group (\(\gD_n\)), tetrahedral group (\(\gT\)), octahedral group (\(\gO\)), and icosahedral group (\(\gI\)). The basic information about these groups can be found in Table 1 in~\cite{Hu_2019}. Using unit quaternions, the symmetric operations of these symmetry groups can be conveniently calculated from no more than two generators, as introduced by Conway et al.~\cite{On_Quaternions_and_Octonions}. Based on Conway's method, generators of these molecular symmetry groups under the orientation conventions in cryo-EM can be derived~\cite{Hu_2019}. The cyclic group is the simplest one and has only one generator. The other four groups require two generators.

\section{Irreducible representations of molecular symmetry groups}

	In this section, we briefly discuss unitary irreducible representations of molecular symmetry groups, including $\gC_n,\gD_n,\gT,\gO,\gI$. We will use unit quaternion description for rotations and the same setting as Table 1 \cite{Hu_2019} in the sequel. Note that every molecular symmetry group can be generated by at most two elements. To determine a representation, it is enough to give its values to the generators.

	The cyclic group $\gC_n$ is generated by an $n$-fold rotation $\sigma=(\cos\frac{\pi}{n},0,0,\sin\frac{\pi}{n})^\top$. There are $n$ irreducible reprentations of dimension $1$:
	\[ \rho_k:\sigma\mapsto e^{\frac{2\pi\sqrt{-1}k}{n}},\quad k=0,1,2,\cdots,n-1. \]
	We refer readers to section 5.1 of \cite{Serre_1977_Representation_Theory} for details.

	The dihedral group $\gD_n$ is generated by an $n$-fold rotation $\sigma=(\cos\frac{\pi}{n},0,0,\sin\frac{\pi}{n})^\top$ and a flip $\tau=(0,1,0,0)^\top$. If $n$ is odd, then there are $\frac{n+3}{2}$ irreducible representations:
	\[ \begin{aligned}
		& \rho_0:\sigma\mapsto 1,\tau\mapsto 1, \\
		& \rho_1:\sigma\mapsto 1,\tau\mapsto-1, \\
		& \rho_{k+1}:\sigma\mapsto\begin{bmatrix}\cos\frac{2\pi k}{n} & -\sin\frac{2\pi k}{n} \\ \sin\frac{2\pi k}{n} & \cos\frac{2\pi k}{n}\end{bmatrix},\tau\mapsto\begin{bmatrix}1 \\ & -1\end{bmatrix}, \quad k=1,2,\cdots,\frac{n-1}{2}.
	\end{aligned} \]
	If $n$ is even, then there are \(\frac{n}{2}+3\) irreducible representations:
	\[ \begin{aligned}
		& \rho_0:\sigma\mapsto 1,\tau\mapsto 1, \\
		& \rho_1:\sigma\mapsto 1,\tau\mapsto-1, \\
		& \rho_2:\sigma\mapsto-1,\tau\mapsto 1, \\
		& \rho_3:\sigma\mapsto-1,\tau\mapsto-1, \\
		& \rho_{k+3}:\sigma\mapsto\begin{bmatrix}\cos\frac{2\pi k}{n} & -\sin\frac{2\pi k}{n} \\ \sin\frac{2\pi k}{n} & \cos\frac{2\pi k}{n}\end{bmatrix},\tau\mapsto\begin{bmatrix}1 \\ & -1\end{bmatrix} \quad k=1,2,\cdots,\frac{n}{2}-1.
	\end{aligned} \]
	We refer readers to section 5.3 of \cite{Serre_1977_Representation_Theory} for details.

	$\gT,\gO,\gI$ are the symmetry groups of regular polytopes, and they are naturally isomorphic to $\gA_4,\gS_4,\gA_5$, respectively \cite{On_Quaternions_and_Octonions}. We refer readers to sections 5.7-5.9 of \cite{Serre_1977_Representation_Theory} and sections 2.3 and 3.1 of \cite{Fulton_2013_Representation_Theory} for the representation theory of these groups and only list the results for short. Be careful that some representations of $\gT,\gO,\gI$ in the following list are not unitary, which are not suitable for the NUG framework. We will then propose a numeric algorithm that produces an equivalent orthogonal representation of a given representation.

	The tetrahedral group $\gT$ is generated by $\sigma=(\frac{1}{2},0,0,\frac{\sqrt{3}}{2})^\top$ and $\tau=(0,0,\frac{\sqrt{6}}{3},\frac{\sqrt{3}}{3})^\top$. There are 4 irreducible representations:
	\[ \begin{aligned}
		& \rho_0:\sigma\mapsto 1,\tau\mapsto 1, \\
		& \rho_1:\sigma\mapsto e^{\frac{2\pi\sqrt{-1}}{3}},\tau\mapsto 1, \\
		& \rho_2:\sigma\mapsto e^{\frac{4\pi\sqrt{-1}}{3}},\tau\mapsto 1, \\
		& \rho_3:\sigma\mapsto \begin{bmatrix}-1 & -1 & -1 \\ 1 & 0 & 0 \\ 0 & 0 & 1\end{bmatrix},\tau\mapsto\begin{bmatrix}-1 & -1 & -1 \\ 0 & 0 & 1 \\ 0 & 1 & 0\end{bmatrix}.
	\end{aligned} \]

	The octahedral group $\gO$ is generated by $\sigma=(\frac{\sqrt{2}}{2},0,0,\frac{\sqrt{2}}{2})^\top$ and $\tau=(\frac{1}{2},\frac{1}{2},\frac{1}{2},\frac{1}{2})^\top$. There are 5 irreducible representations:
	\[ \begin{aligned}
		& \rho_0:\sigma\mapsto 1,\tau\mapsto 1, \\
		& \rho_1:\sigma\mapsto-1,\tau\mapsto 1, \\
		& \rho_2:\sigma\mapsto\begin{bmatrix}1 & 0 \\ -1 & -1\end{bmatrix},\tau\mapsto\begin{bmatrix}0 & 1 \\ -1 & -1\end{bmatrix}, \\
		& \rho_3:\sigma\mapsto\begin{bmatrix}-1 & -1 & -1 \\ 1 & 0 & 0 \\ 0 & 1 & 0\end{bmatrix},\tau\mapsto\begin{bmatrix}0 & 1 & 0 \\ -1 & -1 & -1 \\ 0 & 0 & 1\end{bmatrix}, \\
		& \rho_4:\sigma\mapsto\begin{bmatrix}1 & 1 & 1 \\ -1 & 0 & 0 \\ 0 & -1 & 0\end{bmatrix},\tau\mapsto\begin{bmatrix}0 & 1 & 0 \\ -1 & -1 & -1 \\ 0 & 0 & 1\end{bmatrix}.
	\end{aligned} \]

	The icosahedral group $\gO$ is generated by $\sigma=(0,0,0,1)^\top$ and $\tau=(\frac{1}{2},0,\frac{\sqrt{5}-1}{4},\frac{\sqrt{5}+1}{4})^\top$. There are 5 irreducible representations:
	\[ \begin{aligned}
		& \rho_0:\sigma\mapsto 1,\tau\mapsto 1, \\
		& \rho_1:\sigma\mapsto\begin{bmatrix}-1 & -1 & -1 & -1 \\ 0 & 0 & 1 & 0 \\ 0 & 1 & 0 & 0 \\ 0 & 0 & 0 & 1\end{bmatrix},\tau\mapsto\begin{bmatrix}0 & 0 & 0 & 1 \\ 1 & 0 & 0 & 0 \\ 0 & 0 & 1 & 0 \\ 0 & 1 & 0 & 0\end{bmatrix}, \\
		& \rho_2:\sigma\mapsto\begin{bmatrix}0 & 1 & 0 & 0 & 0 \\ 1 & 0 & 0 & 0 & 0 \\ 0 & 0 & 1 & 0 & 0 \\ -1 & -1 & 0 & -1 & 0 \\ 0 & 1 & 0 & 1 & 1\end{bmatrix},\tau\mapsto\begin{bmatrix}0 & 0 & 1 & 0 & 0 \\ 0 & -1 & -1 & -1 & -1 \\ -1 & 0 & 0 & 0 & 1 \\ 0 & 1 & 0 & 0 & 0 \\ 1 & 0 & 1 & 0 & 0\end{bmatrix}, \\
		& \rho_{k+2}:\sigma\mapsto\begin{bmatrix}-x_k & 1 & -x_k \\ x_k & x_k & -1 \\ 0 & 0 & -1\end{bmatrix},\tau\mapsto\begin{bmatrix}0 & 1+x_k & -1-x_k \\ -1 & -1 & x_k \\ -x_k & -x_k & 1\end{bmatrix} \quad k=1,2. \\
	\end{aligned} \]
	where $x_1,x_2=\frac{-1\pm\sqrt{5}}{2}$.

	Finally, we briefly describe how to find an equivalent unitary representation for a given representation. Let $\gG$ be a finite group and $\rho:\gG\to\operatorname{GL}(n,\sC)$ be a representation. Let $\mA=\sum_{\vg\in\gG}\rho(\vg)^\HT\rho(\vg)$ and find its Cholesky decomposition $\mA=\mP^\HT\mP$. Then $\mP\rho(\vg)\mP^{-1}$ is unitary for all $\vg\in\gG$. We refer readers to proposition 1.5 in~\cite{Fulton_2013_Representation_Theory} for more details.

\section{Evaluation of the set in Appendix 1}

	Our goal is to find the set
	\[ \begin{aligned}
		\{ (L^\orthgroup, \tilde{L}^\orthgroup) \big|
		& L^\orthgroup=2(3-(\sigma_1+\sigma_2+\epsilon\sigma_3)), \\
		& \tilde{L}^\orthgroup=3-(\sigma_1^2+\sigma_2^2+\sigma_3^2) \\
		& 1\geq\sigma_1\geq\sigma_2\geq\sigma_3\geq0, \\
		& \sigma_1+\sigma_2-\epsilon\sigma_3\leq 1, \\
		& \epsilon=\pm1, \}
	\end{aligned} \]
	This is reduced to evaluate
	\[ \begin{aligned}
		\{ (\sigma_1+\sigma_2+\epsilon\sigma_3,\sigma_1^2+\sigma_2^2+\sigma_3^2) \big|
		& 1\geq\sigma_1\geq\sigma_2\geq\sigma_3\geq0, \\
		& \sigma_1+\sigma_2-\epsilon\sigma_3\leq 1, \\
		& \epsilon=\pm1 \}.
	\end{aligned} \]

	In the case that $\epsilon=1$, the range of $\sigma_1 + \sigma_2 + \sigma_3$ is $[0,3]$. When $\sigma_1+\sigma_2+\sigma_3=a$, the minimum of $\sigma_1^2+\sigma_2^2+\sigma_3^2$ is $\frac{a^2}{3}$ when $\sigma_1=\sigma_2=\sigma_3=\frac{a}{3}$. In order to reach the maximum, $\sigma_1$ should be as large as possible. If $a\leq 1$, then $\sigma_1=a$ and it follows $\sigma_2=\sigma_3=0$, and the maximum value is $a^2$. If $1\leq a\leq 3$, then $\sigma_1=1$. Since $\sigma_1+\sigma_2-\sigma_3\leq 1$, we have $\sigma_2\leq\sigma_3$. Then $\sigma_2=\sigma_3=\frac{a-1}{2}$ follows. The maximum is $1^2+2(\frac{a-1}{2})^2=\frac{1}{2}a^2-a+\frac{3}{2}$. In conclusion, when $\epsilon=1$, the range of all possible $(\sigma_1+\sigma_2+\epsilon\sigma_3,\sigma_1^2+\sigma_2^2+\sigma_3^2)$ is
	\[ \begin{aligned}
		& \{(a,b)|0\leq a\leq 1,\frac{a^2}{3}\leq b\leq a^2\} \cup \\
		& \{(a,b)|1\leq a\leq 3,\frac{a^2}{3}\leq b\leq \frac{1}{2}a^2-a+\frac{3}{2}\}.
	\end{aligned} \]

	In the case that $\epsilon=-1$, $\sigma_1+\sigma_2-\sigma_3\leq\sigma_1+\sigma_2+\sigma_3\leq 1$. Let $\sigma_1+\sigma_2-\sigma_3=a$ where $a\in[0,1]$, and the minimum and the maximum of $\sigma_1^2+\sigma_2^2+\sigma_3^2$ are to be investigated. $\sigma_3=0$, in order to reach the minimum. Otherwise, $\sigma_2$ can be decreased along with $\sigma_3$. Therefore, $\sigma_1=\sigma_2=\frac{a}{2}$, and the minimum is $\frac{a^2}{2}$.

	Next, we need to find the maximum. We first note that if $\sigma_1+\sigma_2+\sigma_3<1$, then $\sigma_1=\sigma_2=\sigma_3$ must hold, otherwise we can increase $\sigma_1,\sigma_3$ simultaneously if $\sigma_3<\sigma_2$ or increase $\sigma_2,\sigma_3$ simultaneously if $\sigma_2<\sigma_1$. If $\sigma_1=\sigma_2=\sigma_3$, then by $\sigma_1+\sigma_2-\sigma_3=a$ they are all equal to $a$. This case can happen if and only if $\sigma_1+\sigma_2+\sigma_3\leq 1$, i.e., $a\leq\frac{1}{3}$. In this case, the maximum is $3a^2$. Otherwise, we have $\sigma_1+\sigma_2+\sigma_3=1$. It follows that $\sigma_1+\sigma_2=\frac{1+a}{2}$ and $\sigma_3=\frac{1-a}{2}$. We note that this case is only possible when the maximum value of $\sigma_2$, i.e., $\frac{1+a}{4}$, is greater than or equal to $\sigma_3=\frac{1-a}{2}$. Hence $a\geq\frac{1}{3}$. Moreover, to achieve maximum in this case, $\sigma_2$ should be as small as possible. The smallest possible value of $\sigma_2$ is $\sigma_3=\frac{1-a}{2}$, and $\sigma_1=a$ follows. Note that $\sigma_1\geq\sigma_2$ holds in this case. Therefore, the maximum is $a^2+2(\frac{1-a}{2})^2=\frac{3}{2}a^2-a+\frac{1}{2}$. In conclusion, when $\epsilon=-1$, the range of all possible $(\sigma_1+\sigma_2+\epsilon\sigma_3,\sigma_1^2+\sigma_2^2+\sigma_3^2)$ is
	\[ \begin{aligned}
		& \{(a,b)|0\leq a\leq\frac{1}{3}, \frac{a^2}{2}\leq b\leq 3a^2\} \cup \\
		& \{(a,b)|\frac{1}{3}\leq a\leq 1,\frac{a^2}{2}\leq b\leq \frac{3}{2}a^2-a+\frac{1}{2}\}.
	\end{aligned} \]
	Combining the previous results we obtain
	\[ \begin{aligned}
		& \{(\sigma_1+\sigma_2+\epsilon\sigma_3,\sigma_1^2+\sigma_2^2+\sigma_3^2)\} \\
		= & \{(a,b)|0\leq a\leq\frac{1}{3}, \frac{a^2}{3}\leq b\leq 3a^2\} \cup \\
		& \{(a,b)|\frac{1}{3}\leq a\leq 1,\frac{a^2}{3}\leq b\leq \frac{3}{2}a^2-a+\frac{1}{2}\} \cup \\
		& \{(a,b)|1\leq a\leq 3,\frac{a^2}{3}\leq b\leq \frac{1}{2}a^2-a+\frac{3}{2}\}.
	\end{aligned} \]
	Put this range into original expression and the result follows.

\section{Sensitivity analysis of hyperparameters $m$ and $c$}

	In our rounding algorithm, we have identified two hyperparameters: the capacity $m$ and the threshold $c$. In Section 4.1.2 and Table 4, we demonstrate how the performance of our rounding method changes with varying $m$ and $c$. However, due to space limitations, we could only present a subset of the results in the previous section. To provide a more comprehensive analysis, we conducted additional experiments with a wider range of hyperparameter values.

	Specifically, in this section, we present the detailed results of our experiments. We tested the values of
    \begin{align}
        m & \in \{4, 8, 12, 16, 20, 24, 28, 32, 36, 40\}, \\
        c & \in \{0, 0.1, 0.2, 0.3, 0.4, 0.5, 0.6, 0.7, 0.8, 0.9\}.
    \end{align}
    The comprehensive results are provided in the following two tables.

	\begin{landscape}
		\begin{table}[!htbp]
			\centering
			\caption{The (Accuracy, Maximal RCG-NUG) results for varying $m$ under the $d_\orthgroup^A$.}
			\begin{tabular}{c|ccccc}
				\toprule
				\multirow{2}*{Group $\gG$} & \multicolumn{5}{c}{$c=0.99$} \\
				& $m=4$ & $m=8$ & $m=12$ & $m=16$ & $m=20$ \\
				\midrule
				$\gC_2$ & (98.0\%, 0.82\%) & (98.1\%, 0.82\%) & (98.1\%, 0.82\%) & (98.1\%, 0.82\%) & (98.1\%, 0.82\%) \\
				$\gC_7$ & (97.9\%, 8.27\%) & (99.4\%, 2.16\%) & (99.8\%, 0.50\%) & (99.9\%, 0.22\%) & (99.9\%, 0.22\%) \\
				$\gD_2$ & (99.9\%, 0.13\%) & (99.9\%, 0.13\%) & (99.9\%, 0.13\%) & (99.9\%, 0.13\%) & (99.9\%, 0.13\%) \\
				$\gD_7$ & (99.8\%, 1.62\%) & (100\%,     0\%) & (100\%,     0\%) & (100\%,     0\%) & (100\%,     0\%) \\
				$\gT$   & (100\%,     0\%) & (100\%,     0\%) & (100\%,     0\%) & (100\%,     0\%) & (100\%,     0\%) \\
				$\gO$   & (100\%,     0\%) & (100\%,     0\%) & (100\%,     0\%) & (100\%,     0\%) & (100\%,     0\%) \\
				$\gI$   & (100\%,     0\%) & (100\%,     0\%) & (100\%,     0\%) & (100\%,     0\%) & (100\%,     0\%) \\
				\midrule
				\multirow{2}*{Group $\gG$} & \multicolumn{5}{c}{$c=0.99$} \\
				& $m=24$ & $m=28$ & $m=32$ & $m=36$ & $m=40$ \\
				\midrule
				$\gC_2$ & (98.1\%, 0.82\%) & (98.1\%, 0.82\%) & (98.1\%, 0.82\%) & (98.1\%, 0.82\%) & (98.1\%, 0.82\%) \\
				$\gC_7$ & (99.9\%, 0.22\%) & (100\%,     0\%) & (100\%,     0\%) & (100\%,     0\%) & (100\%,     0\%) \\
				$\gD_2$ & (99.9\%, 0.13\%) & (99.9\%, 0.13\%) & (99.9\%, 0.13\%) & (99.9\%, 0.13\%) & (99.9\%, 0.13\%) \\
				$\gD_7$ & (100\%,     0\%) & (100\%,     0\%) & (100\%,     0\%) & (100\%,     0\%) & (100\%,     0\%) \\
				$\gT$   & (100\%,     0\%) & (100\%,     0\%) & (100\%,     0\%) & (100\%,     0\%) & (100\%,     0\%) \\
				$\gO$   & (100\%,     0\%) & (100\%,     0\%) & (100\%,     0\%) & (100\%,     0\%) & (100\%,     0\%) \\
				$\gI$   & (100\%,     0\%) & (100\%,     0\%) & (100\%,     0\%) & (100\%,     0\%) & (100\%,     0\%) \\
				\bottomrule
			\end{tabular}
		\end{table}
	\end{landscape}

	\begin{landscape}
		\begin{table}[!htbp]
			\centering
			\caption{The (Accuracy, Maximal RCG-NUG) results for varying $c$ under the $d_\orthgroup^A$.}
			\begin{tabular}{c|ccccc}
				\toprule
				\multirow{2}*{Group $\gG$} & \multicolumn{5}{c}{$m=20$} \\
				& $c=0$ & $c=0.1$ & $c=0.2$ & $c=0.3$ & $c=0.4$ \\
				\midrule
				$\gC_2$ & (89.6\%,  2.40\%) & (89.6\%,  2.40\%) & (89.6\%,  2.40\%) & (89.6\%,  2.40\%) & (89.6\%,  2.40\%) \\
				$\gC_7$ & (95.3\%,  2.48\%) & (95.3\%,  2.48\%) & (95.3\%,  2.48\%) & (95.3\%,  2.48\%) & (95.4\%,  2.48\%) \\
				$\gD_2$ & (94.6\%,  8.74\%) & (94.6\%,  8.74\%) & (94.6\%,  8.74\%) & (94.6\%,  8.74\%) & (94.6\%,  8.74\%) \\
				$\gD_7$ & (96.5\%, 17.62\%) & (96.5\%, 17.62\%) & (96.5\%, 17.62\%) & (96.5\%, 17.62\%) & (96.5\%, 17.62\%) \\
				$\gT$   & (98.2\%, 33.32\%) & (98.2\%, 33.32\%) & (98.2\%, 33.32\%) & (98.2\%, 33.32\%) & (98.2\%, 33.32\%) \\
				$\gO$   & (99.2\%,  6.89\%) & (99.2\%,  6.89\%) & (99.2\%,  6.89\%) & (99.2\%,  6.89\%) & (99.2\%,  6.89\%) \\
				$\gI$   & (99.7\%,  2.00\%) & (99.7\%,  2.00\%) & (99.7\%,  2.00\%) & (99.7\%,  2.00\%) & (99.7\%,  2.00\%) \\
				\midrule
				\multirow{2}*{Group $\gG$} & \multicolumn{5}{c}{$m=20$} \\
				& $c=0.5$ & $c=0.6$ & $c=0.7$ & $c=0.8$ & $c=0.9$ \\
				\midrule
				$\gC_2$ & (89.6\%,  2.40\%) & (89.6\%,  2.40\%) & (89.6\%,  2.40\%) & (90.3\%, 1.30\%) & (93.7\%, 1.30\%) \\
				$\gC_7$ & (95.7\%,  2.48\%) & (97.0\%,  2.24\%) & (98.3\%,  0.88\%) & (99.1\%, 0.40\%) & (99.7\%, 0.40\%) \\
				$\gD_2$ & (94.6\%,  8.74\%) & (94.7\%,  8.74\%) & (95.6\%,  7.72\%) & (97.4\%, 2.85\%) & (99.5\%, 1.14\%) \\
				$\gD_7$ & (96.8\%, 13.11\%) & (97.4\%, 13.11\%) & (98.3\%, 12.51\%) & (99.3\%, 0.84\%) & (99.4\%, 0.60\%) \\
				$\gT$   & (98.2\%, 33.32\%) & (99.0\%,  7.28\%) & (99.8\%,  2.24\%) & (99.9\%, 1.81\%) & (99.9\%, 1.81\%) \\
				$\gO$   & (99.2\%,  6.89\%) & (99.8\%,  0.18\%) & (100\%,      0\%) & (100\%,     0\%) & (100\%,     0\%) \\
				$\gI$   & (99.7\%,  2.00\%) & (99.8\%,  2.00\%) & (100\%,      0\%) & (100\%,     0\%) & (100\%,     0\%) \\
				\bottomrule
			\end{tabular}
		\end{table}
	\end{landscape}

\end{document}